\documentclass{article}
\usepackage{amsmath}
\usepackage{amsfonts}
\usepackage{hyperref}

\newtheorem{theorem}{Theorem}

\newtheorem{corollary}[theorem]{Corollary}

\newtheorem{definition}[theorem]{Definition}
\newtheorem{example}[theorem]{Example}

\newtheorem{remark}[theorem]{Remark}

\newenvironment{proof}[1][Proof]{\noindent\textbf{#1.} }{\ \rule{0.5em}{0.5em}}

\begin{document}

\title{New Laplace convolution integrals involving exponential, error, and
parabolic cylinder functions with applications in heat transfer and linear
viscoelasticity}
\author{Gonz\'{a}lez Santander, Juan Luis \\
Department of Mathematics. Universisty of Oviedo. \\
C/ Calvo Sotelo 18. Oviedo 33007, Spain. }
\maketitle

\begin{abstract}
We compute several new Laplace convolution integrals involving exponential
functions, error functions, and parabolic cylinder functions. We extend our
results using formulas for repeated integration and differentiation. We
apply these convolution integrals to heat transfer and linear viscoelastic
problems. As a by-product, we obtain new integral representations for the
error function and the complementary error function.
\end{abstract}

$\mathbf{Keywords}$: Laplace transform, convolution theorem, error
functions, generalized hypergeometric functions, heat transfer problems,
linear viscoelasticity problems.

$\mathbf{MSC}$: 44A10, 33C20, 33B20

\section{Introduction}

The Laplace transform is one of the most powerful analytical tools in
operational calculus, since it provides systematic techniques for solving
linear differential equations \cite{Davies}. However, the Laplace transform
can also be applied to the solution of nonlinear equations, such as those
arising in boundary layer fluid mechanics problems \cite%
{ApelblatBlasius1,ApelblatBlasius2,ApelblatBlasius3}. Moreover, it can be
applied to fractional differential equations, as in the case of the
analytical solution of the Bagley--Torvik equation \cite{BagleyTorvik}.

Despite the fact that classical compilations of Laplace and inverse Laplace
transforms are very extensive (see \cite{Oberhettinger,Prudnikov4,Prudnikov5}%
), it is worth noting that new Laplace transforms and their applications
continue to appear in the recent literature (see, for example, \cite%
{Rogosin,ApelblatIntegralML,LaplaceJL}).

Here, we focus on the application of the Laplace transform to derive new
integrals that are not reported in the most common tables, such as \cite%
{Prudnikov1,Prudnikov2,Brychkov}. In particular, we consider the application
of the convolution theorem for the Laplace transform. In \cite{Narahari},
the convolution theorem was applied to obtain new inverse Laplace transforms
from known integrals. However, here we are interested in the reciprocal
approach, namely, obtaining new integrals from known inverse Laplace
transforms by applying the convolution theorem. This approach has previously
been used by the author in the context of modified Bessel functions \cite%
{IntegralesLaplaceJL}. However, it seems that the convolution theorem has
not yet been systematically exploited to compute convolution integrals using
the Laplace transform. Moreover, convolution integrals appear in many
applied problems. For instance, the application of Duhamel's principle to
heat transfer problems leads to convolution integrals \cite[Sect. 1.14.II]%
{Jaeger}. Also, applying Boltzmann superposition principle to linear
viscoelastic materials, the general stress-strain relation is expressed in
terms of a convolution integral \cite[Sect. 2.1]{MainardiBook}.

Since the number of such convolution integrals is quite large, we have
structured the material as a series of papers, depending on the type of
function appearing in the integrand. This paper is the second in the series;
the first one was devoted to Laplace convolution integrals involving
trigonometric and hyperbolic functions \cite{ConvolutionTrigJL}.

This paper is organized as follows. In Section \ref{Section: Preliminaries},
we present some fundamentals of the Laplace integral transform, as well as
the formula for repeated integration. We also define all the special
functions used throughout the paper, together with some of the formulas they
satisfy. Section \ref{Section: Main Results} is devoted to the derivation of
Laplace convolution integrals involving exponential, error, and parabolic
cylinder functions. We show how these convolution integrals arise in heat
transfer and linear viscoelastic problems, enhancing the known analytical
solutions to these kinds of problems. Finally, we present our conclusions in
Section \ref{Section: Conclusions}.

\section{Preliminaries \label{Section: Preliminaries}}

\subsection{Integral transforms}

\begin{definition}[Laplace transform]
Let $f\left( t\right) $ be a real- or complex-valued function of the
variable $t>0$, and let $s$ be a real or complex parameter. The Laplace
transform of $f\left( t\right) $ is defined as \cite[Eqn. 1.1]{Schiff}:
\begin{equation}
F\left( s\right) =\mathcal{L}\left[ f\left( t\right) ;s\right]
=\int_{0}^{\infty }e^{-st}f\left( t\right) \,dt, =\lim_{\tau \rightarrow
\infty }\int_{0}^{\tau }e^{-st}f\left( t\right) \,dt,  \label{Laplace_def}
\end{equation}%
whenever the limit exists and is finite.
\end{definition}

We recall that if $f\left( t\right) $ is piecewise continuous on $\left[
0,\infty \right) $ and of exponential order $\alpha $, then the Laplace
transform $\mathcal{L}\left[ f\left( t\right) ;s\right] $ exists for $\Re
\left( s\right) >\alpha $ and converges absolutely. In this case, we denote
the inverse Laplace transform by
\begin{equation*}
f\left( t\right) =\mathcal{L}^{-1}\left[ F\left( s\right) ;t\right] .
\end{equation*}

\begin{theorem}[Derivative theorem]
If $f\left( t\right) $ is continuous on $\left( 0,\infty \right) $, is of
exponential order $\alpha $, and $f^{\prime }\left( t\right) $ is piecewise
continuous on $\left[ 0,\infty \right) $, then \cite[Eqn. 2.17]{Schiff}:
\begin{equation}
\mathcal{L}\left[ f^{\prime }\left( t\right) ;s\right] =s\,\mathcal{L}\left[
f\left( t\right) ;s\right] -f\left( 0^{+}\right) ,\quad \Re \left( s\right)
>\alpha .  \label{Derivative_Laplace}
\end{equation}
\end{theorem}

\begin{theorem}[Convolution]
If $f\left( t\right) =\mathcal{L}^{-1}\left[ F\left( s\right) ;t\right] $
and $g\left( t\right) =\mathcal{L}^{-1}\left[ G\left( s\right) ;t\right] $
are piecewise continuous on $\left[ 0,\infty \right) $ and of exponential
order $\alpha $, then \cite[Theorem 2.39]{Schiff}:
\begin{equation}
\int_{0}^{t}\mathcal{L}^{-1}\left[ F\left( s\right) ;\tau \right] \,\mathcal{%
L}^{-1}\left[ G\left( s\right) ;t-\tau \right] \,d\tau =\mathcal{L}^{-1}%
\left[ F\left( s\right) \,G\left( s\right) ;t\right] , \quad \Re \left(
s\right) >\alpha ,\,t>0.  \label{Convolution_Laplace}
\end{equation}
\end{theorem}

Although the convolution theorem (\ref{Convolution_Laplace}) is stated for $%
t>0$, it is worth noting that, in many cases, this restriction can be
relaxed to allow negative or complex values of $t$. However, throughout the
paper, we assume that $t>0$.

\begin{theorem}[Repeated integration]
If $f\left( t\right) $ is continuous or piecewise continuous on $\left[a,b%
\right] $, then \cite[Eqn. 1.4.31]{NIST}:
\begin{equation}
\int_{a}^{b}dt_{n}\int_{a}^{t_{n}}dt_{n-1}\cdots
\int_{a}^{t_{2}}dt_{1}\int_{a}^{t_{1}}f\left( t\right) dt=\frac{1}{n!}%
\int_{a}^{b}\left( b-t\right) ^{n}\,f\left( t\right) \,dt.
\label{Repeated_integration}
\end{equation}
\end{theorem}

\subsection{Special functions}

For $p\leq q$, the following series converges for $z\in \mathbb{C}$ and
defines an entire function called the generalized hypergeometric function
\cite[Eqn. 16.2.1]{NIST}:
\begin{equation}
_{p}F_{q}\left(
\begin{array}{c}
a_{1},\ldots ,a_{p} \\
b_{1},\ldots ,b_{q}%
\end{array}%
;z\right) =\sum_{k=0}^{\infty }\frac{\left( a_{1}\right) _{k}\cdots \left(
a_{p}\right) _{k}}{\left( b_{1}\right) _{k}\cdots \left( b_{q}\right) _{k}}%
\frac{z^{k}}{k!},  \label{pFq_def}
\end{equation}%
where $\left( x\right) _{k}$ denotes the Pochhammer symbol \cite[Eqn. 5.2.4-5%
]{NIST}%
\begin{equation}
\left( x\right) _{k}=\left\{
\begin{array}{ll}
x\left( x+1\right) \left( x+2\right) \cdots \left( x+k-1\right) , & \left(
x\right) _{0}=1, \\
\displaystyle%
\frac{\Gamma \left( x+k\right) }{\Gamma \left( x\right) },\quad & x\neq
0,-1,-2,\ldots%
\end{array}%
\right.  \label{Pochhammer_def}
\end{equation}%
and $\Gamma \left( z\right) $ denotes the gamma function \cite[Eqn. 5.2.1]%
{NIST},%
\begin{equation}
\Gamma \left( z\right) =\int_{0}^{\infty }t^{z-1}e^{-t}dt,\quad \Re \left(
z\right) >0.  \label{Gamma_def}
\end{equation}%
A well-known property of the gamma function is \cite[Eqn. 1.2.1]{Lebedev}
\begin{equation}
\Gamma \left( z+1\right) =z\,\Gamma \left( z\right) .
\label{Gamma_factorial}
\end{equation}%
The Pochhammer symbol has the following properties \cite[Eqn. 18:5:1\&7]%
{Atlas}:%
\begin{eqnarray}
\left( -x\right) _{n} &=&\left( -1\right) ^{n}\left( x-n+1\right) _{n},
\label{Pochhammer_1} \\
x\left( x+1\right) _{n} &=&\left( x+n\right) \left( x\right) _{n}.
\label{Pochhammer_2}
\end{eqnarray}%
Well-known particular values of the gamma function are \cite[Eqns. 5.4.1\&6]%
{NIST}:
\begin{equation}
\Gamma \left( 1\right) =1,\quad \Gamma \left( \frac{1}{2}\right) =\sqrt{\pi }%
.  \label{Gamma_particular}
\end{equation}%
Moreover, according to \cite[Eqns. 43:4:1\&3]{Atlas}, we have for $%
n=0,1,2,\ldots $
\begin{eqnarray}
\Gamma \left( n+1\right) &=&n!  \label{Gamma(n+1)} \\
\Gamma \left( \frac{1}{2}+n\right) &=&\frac{\left( 2n-1\right) !!}{2^{n}}%
\sqrt{\pi }.  \label{gamma(n+1/2)}
\end{eqnarray}

The beta function is defined as \cite[Eqn. 5.12.1]{NIST}:
\begin{equation}
\mathrm{B}\left( a,b\right) =\int_{0}^{1}t^{a-1}\left( 1-t\right) ^{b-1}dt=%
\frac{\Gamma \left( a\right) \,\Gamma \left( b\right) }{\Gamma \left(
a+b\right) },\quad \Re \left( a\right) ,\Re \left( b\right) >0.
\label{Beta_def}
\end{equation}

Also, a well-known summation formula is the Chu-Vandermonde identity:%
\begin{equation}
_{2}F_{1}\left(
\begin{array}{c}
-n,b \\
c%
\end{array}%
;1\right) =\frac{\left( c-b\right) _{n}}{\left( c\right) _{n}},\quad
n=0,1,2,\ldots  \label{Chu_formula}
\end{equation}%
In addition, a useful transformation of the Kummer function is \cite[Eqn.
7.11.1(2)]{Prudnikov3}:
\begin{equation}
_{1}F_{1}\left(
\begin{array}{c}
a \\
b%
\end{array}%
;z\right) =e^{z}\,_{1}F_{1}\left(
\begin{array}{c}
b-a \\
b%
\end{array}%
;-z\right) .  \label{1F1_reduction_1}
\end{equation}

The error function, the complementary error function, and the imaginary
error function are defined as \cite[Eqns. 40:3:1\&40:0:1\&42:1:1]{Atlas}:
\begin{eqnarray}
\mathrm{erf}\left( z\right) &=&\frac{2}{\sqrt{\pi }}\int_{0}^{z}\exp \left(
-t^{2}\right) \,dt,  \label{erf_def} \\
\mathrm{erfc}\left( z\right) &=&\frac{2}{\sqrt{\pi }}\int_{z}^{\infty }\exp
\left( -t^{2}\right) \,dt=1-\mathrm{erf}\left( z\right) ,  \label{erfc_def}
\\
\mathrm{erfi}\left( z\right) &=&\frac{\mathrm{erf}\left( i\,z\right) }{i}.
\label{erfi_def}
\end{eqnarray}%
Directly from (\ref{erf_def}) and (\ref{erfc_def}), we have
\begin{eqnarray}
\mathrm{erf}\left( 0\right) &=&0,  \label{erf(0)} \\
\mathrm{erfc}\left( 0\right) &=&1.  \label{erfc(0)}
\end{eqnarray}%
The error function has the following series expansion \cite[Eqn. 40:6:1]%
{Atlas}:
\begin{equation}
\mathrm{erf}\left( z\right) =\frac{2z}{\sqrt{\pi }}\sum_{k=0}^{\infty }\frac{%
\left( \frac{1}{2}\right) _{k}}{k!\left( \frac{3}{2}\right) _{k}}\left(
-z^{2}\right) ^{k},  \label{erf_series}
\end{equation}%
and the complementary error function has the following asymptotic behaviour
\cite[Eqn. 40:9:1]{Atlas}:%
\begin{equation}
\mathrm{erfc}\left( z\right) \approx \frac{\exp \left( -z^{2}\right) }{\sqrt{%
\pi }z}\sum_{k=0}^{\infty }\frac{\left( \frac{1}{2}\right) _{k}}{\left(
-z\right) ^{k}},\quad z\rightarrow \infty ,  \label{erfc_asymptotic}
\end{equation}%
thus, according to (\ref{erfc_def})\ and (\ref{erfc_asymptotic}), we have \
\begin{equation}
\mathrm{erf}\left( z\right) \approx 1-\frac{\exp \left( -z^{2}\right) }{%
\sqrt{\pi }z},\quad z\rightarrow \infty .  \label{erf_asymptotic}
\end{equation}%
Also, for $n\geq 1$, we have the following derivative formula \cite[Eqn.
1.5.1(1)]{Brychkov}:%
\begin{equation}
\frac{d^{n}}{dz^{n}}\mathrm{erf}\left( a\,z\right) =\left( -1\right) ^{n-1}%
\frac{2\,a^{n}}{\sqrt{\pi }}\exp \left( -a^{2}z^{2}\right) \,H_{n-1}\left(
a\,z\right) ,  \label{D^n[erf]}
\end{equation}%
where $H_{n}\left( z\right) $ denotes the Hermite polynomial of $n$-th order
\cite[Eqn. 4.9.2]{Lebedev}. Moreover, the following derivative formula holds
\cite[Eqn. 10.5.4]{Lebedev}:
\begin{equation}
\frac{d^{n}}{dz^{n}}\left[ \exp \left( z^{2}\right) \,\mathrm{erfc}\left(
z\right) \right] =\frac{2^{n+1}n!\left( -1\right) ^{n}}{\sqrt{\pi }}%
H_{-n-1}\left( z\right) .  \label{D^n[exp*erfc]}
\end{equation}

The Dawson integral is defined as \cite[Eqn. 42:3:1]{Atlas}
\begin{equation}
\mathrm{daw}\left( z\right) =\int_{0}^{z}\exp \left( t^{2}-z^{2}\right) \,dt.
\label{Dawson_def}
\end{equation}%
The following relation between the Dawson integral and the imaginary error
function is satisfied \cite[Eqn. 42:1:1]{Atlas}:
\begin{equation}
\mathrm{daw}\left( z\right) =\frac{\sqrt{\pi }}{2}\exp \left( -z^{2}\right)
\,\mathrm{erfi}\left( z\right) .  \label{Dawson_erfi}
\end{equation}%
In addition, we have the expansion \cite[Eqn. 42:6:1]{Atlas}:\
\begin{equation}
\mathrm{daw}\left( z\right) =z\sum_{k=0}^{\infty }\frac{\left( -z^{2}\right)
^{k}}{\left( \frac{3}{2}\right) _{k}},  \label{Dawson_series}
\end{equation}%
and the asymptotic behaviour \cite[Eqn. 42:6:4]{Atlas}:
\begin{equation}
\mathrm{daw}\left( z\right) \approx \frac{1}{2z}\sum_{k=0}^{\infty }\frac{%
\left( \frac{1}{2}\right) _{k}}{z^{2k}},\quad z\rightarrow \infty .
\label{Dawson_z->inf}
\end{equation}

The integral of the complementary error function is defined as \cite[Eqn.
40:13:1]{Atlas}:%
\begin{equation}
\mathrm{ierfc}\left( z\right) =\int_{z}^{\infty }\mathrm{erfc}\left(
t\right) \,dt.  \label{ierfc_def}
\end{equation}%
The $n$-fold integrals of the complementary error function are recursively
defined as \cite[Eqn. 40:13:3]{Atlas}:
\begin{equation*}
\mathrm{i}^{n}\mathrm{erfc}\left( z\right) =\int_{z}^{\infty }\mathrm{i}%
^{n-1}\mathrm{erfc}\left( t\right) \,dt,\quad n=1,2,\ldots
\end{equation*}%
where \cite[Eqns. 43.13.4\&5]{Atlas},%
\begin{eqnarray}
\mathrm{i}^{0}\mathrm{erfc}\left( z\right) &=&\mathrm{erfc}\left( z\right) ,
\label{i^0_erfc} \\
\mathrm{i}^{1}\mathrm{erfc}\left( z\right) &=&\mathrm{ierfc}\left( z\right) .
\label{i^1_erfc}
\end{eqnarray}

The parabolic cylinder function $D_{\nu }\left( z\right) $ is defined as the
solution of the differential equation \cite[Eqn. 12.2.4]{NIST}:%
\begin{equation*}
\frac{d^{2}w}{dz^{2}}+\left( \nu +\frac{1}{2}-\frac{z^{2}}{2}\right) w=0.
\end{equation*}%
The following reduction formulas hold true \cite[Eqns. 12.7.2\&12.7.6-7]%
{NIST}:
\begin{eqnarray}
D_{-n-1}\left( z\right) &=&\sqrt{\pi }\,2^{\left( n-1\right) /2}\exp \left(
\frac{z^{2}}{4}\right) \,\,\mathrm{i}^{n}\mathrm{erfc}\left( \frac{z}{\sqrt{2%
}}\right) ,\quad n=0,1,2,\ldots  \label{D_-n-1_reduction} \\
D_{n}\left( z\right) &=&2^{-n/2}\exp \left( -\frac{z^{2}}{4}\right)
\,H_{n}\left( \frac{z}{\sqrt{2}}\right) ,\quad n=0,1,2,\ldots
\label{D_n_reduction}
\end{eqnarray}%
thus, for instance,%
\begin{eqnarray}
D_{-1}\left( z\right) &=&\sqrt{\frac{\pi }{2}}\exp \left( \frac{z^{2}}{4}%
\right) \,\mathrm{erfc}\left( \frac{z}{\sqrt{2}}\right) ,
\label{D_-1_resultado} \\
D_{0}\left( z\right) &=&\exp \left( -\frac{z^{2}}{4}\right) ,
\label{D_0_reduction} \\
D_{1}\left( z\right) &=&z\exp \left( -\frac{z^{2}}{4}\right) .
\label{D_1_reduction}
\end{eqnarray}%
Moreover, according to \cite[Eqn. 46:7:1]{Atlas}:%
\begin{equation}
D_{\nu }\left( 0\right) =\frac{2^{\nu /2}\sqrt{\pi }}{\Gamma \left( \frac{%
1-\nu }{2}\right) }.  \label{D_nu(0)}
\end{equation}

For $b\notin
\mathbb{Z}
$, the Tricomi function is defined as \cite[Eqn. 48:3:1]{Atlas}:%
\begin{eqnarray}
&&\mathrm{U}\left( a,b,z\right)   \label{Tricomi_def} \\
&=&\frac{\Gamma \left( b-1\right) }{\Gamma \left( a\right) }%
\,z^{1-b}\,_{1}F_{1}\left(
\begin{array}{c}
a-b+1 \\
2-b%
\end{array}%
;z\right) +\frac{\Gamma \left( 1-b\right) }{\Gamma \left( a-b+1\right) }%
\,_{1}F_{1}\left(
\begin{array}{c}
a \\
b%
\end{array}%
;z\right) ,  \notag
\end{eqnarray}%
An integral representation of the Tricomi function is given by the following
Laplace transform \cite[Eqn. 48:3:6]{Atlas}:
\begin{eqnarray}
&&\mathrm{U}\left( a,b,z\right)   \label{Tricomi_integral} \\
&=&\frac{1}{\Gamma \left( a\right) }\int_{0}^{\infty }e^{-z\,t}t^{a-1}\left(
1+t\right) ^{b-a-1}dt=\frac{1}{\Gamma \left( a\right) }\mathcal{L}\left[
t^{a-1}\left( 1+t\right) ^{b-a-1};z\right] .  \notag
\end{eqnarray}%
A useful reduction formula is \cite[Eqn. 48:4:9]{Atlas}:%
\begin{equation}
\mathrm{U}\left( a+\frac{1}{2},\frac{3}{2},z\right) =\frac{2^{a}}{\sqrt{z}}%
\exp \left( \frac{z}{2}\right) \,D_{-2a}\left( \sqrt{2z}\right) .
\label{Tricomi_reduction}
\end{equation}

The exponential integral is defined as \cite[Eqn. 6.2.1]{NIST}
\begin{equation*}
\mathrm{E}_{1}\left( z\right) =\int_{z}^{\infty }\frac{e^{-t}}{t}dt,\quad
z\neq 0,
\end{equation*}%
and the entire function denoted as the complementary exponential integral is
defined as \cite[Eqn. 6.2.3]{NIST}:
\begin{equation}
\mathrm{Ein}\left( z\right) =\int_{0}^{z}\frac{1-e^{-t}}{t}dt.
\label{Ein_def}
\end{equation}%
The power-series expansion of the complementary exponential integral is
\begin{equation}
\mathrm{Ein}\left( z\right) =\sum_{n=1}^{\infty }\frac{\left( -1\right)
^{n-1}z^{n}}{n! \, n}.  \label{Ein_series}
\end{equation}%
The relation between the exponential integral and the complementary
exponential integral is given by \cite[Eqn. 6.2.4]{NIST}:%
\begin{equation}
\mathrm{Ein}\left( z\right) =\mathrm{E}_{1}\left( z\right) +\log z+\gamma ,
\label{Ein_E1}
\end{equation}%
where $\gamma $ is the Euler-Mascheroni constant. Since the exponential
integral has the asymptotic expansion:
\begin{equation}
\mathrm{E}_{1}\left( z\right) \approx \frac{e^{-z}}{z}\sum_{n=0}^{\infty }%
\frac{\left( -1\right) ^{n}n!}{z^{n}},\quad z\rightarrow \infty ,\
\left\vert \arg z\right\vert <\frac{3\pi }{2},  \label{E1_asymptotic}
\end{equation}%
from (\ref{Ein_E1}) and (\ref{E1_asymptotic}), we obtain the asymptotic
expansion:%
\begin{equation}
\mathrm{Ein}\left( z\right) \approx \log z+\gamma ,\quad z\rightarrow \infty
.  \label{Ein(z)_asymptotic}
\end{equation}

The two-parameter Mittag-Leffler function is defined as \cite[Eqn. 18.1(19)]%
{Bateman}:
\begin{equation}
\mathrm{E}_{\mu ,\nu }\left( z\right) =\sum_{k=0}^{\infty }\frac{z^{k}}{%
\Gamma \left( \mu \,k+\nu \right) },\quad \Re \left( \mu \right) >0.
\label{ML_def}
\end{equation}%
According to \cite[Sect. 45:14]{Atlas},%
\begin{equation}
\mathrm{E}_{1/2,1}\left( z\right) =\exp \left( z^{2}\right) \,\mathrm{erfc}%
\left( -z\right) .  \label{ML_1/2,1}
\end{equation}%
A useful inverse Laplace formula that involves the Mittag-Leffler function
is the following one \cite[Eqn. 45:14:4]{Atlas}:%
\begin{equation}
\mathcal{L}^{-1}\left[ \frac{s^{\mu -\nu }}{s^{\mu }-b};t\right] =t^{\nu -1}%
\mathrm{E}_{\mu ,\nu }\left( b\,t^{\mu }\right) .  \label{L-1[ML]}
\end{equation}%
\qquad

The modified Bessel function is defined as \cite[Eqn. 5.7.1]{Lebedev}:%
\begin{equation}
I_{\nu }\left( z\right) =\sum_{k=0}^{\infty }\frac{\left( z/2\right) ^{\nu
+2k}}{k!\,\Gamma \left( k+\nu +1\right) },\quad \left\vert z\right\vert
<\infty ,\quad \left\vert \arg z\right\vert <\pi ,  \label{I_Bessel_def}
\end{equation}%
and the Macdonald function is defined as \cite[Eqn. 5.7.2]{Lebedev}:
\begin{equation}
K_{\nu }\left( z\right) =\frac{\pi }{2}\left( \frac{I_{-\nu }\left( z\right)
-I_{\nu }\left( z\right) }{\sin \pi \nu }\right) ,\quad \left\vert \arg
z\right\vert <\pi ,\quad \nu \neq 0,\pm 1,\pm 2,\ldots  \label{K_Bessel_def}
\end{equation}%
From the definition of the Macdonald function, it is clear that%
\begin{equation}
K_{-\nu }\left( z\right) =K_{\nu }\left( z\right) .  \label{K_nu=K_-nu}
\end{equation}

Finally, the following reduction formulas hold true \cite[Eqn.
7.11.2(10)\&(11)]{Prudnikov3}:%
\begin{eqnarray}
_{1}F_{1}\left(
\begin{array}{c}
\frac{1}{2} \\
1%
\end{array}%
;z\right) &=&e^{z/2}\,I_{0}\left( \frac{z}{2}\right) ,
\label{1F1_reduction_3} \\
_{1}F_{1}\left(
\begin{array}{c}
\frac{1}{2} \\
\frac{3}{2}%
\end{array}%
;z\right) &=&\frac{1}{2}\sqrt{\frac{\pi }{z}}\,\mathrm{erfi}\left( \sqrt{z}%
\right) ,  \label{1F1_reduction_4b}
\end{eqnarray}

\section{Main results \label{Section: Main Results}}

\subsection{Exponential function}

\begin{theorem}
\label{Theorem: Laplace 1} For $\Re \left( a^{2}\right) >0 $ and $\Re \left(
b^{2}\right) >0$, the following convolution integral holds:%
\begin{equation}
\int_{0}^{t}\frac{\exp \left( -\frac{a^{2}}{\tau }-\frac{b^{2}}{t-\tau }%
\right) \,}{\sqrt{\tau \left( t-\tau \right) }}\,d\tau =\pi \,\mathrm{erfc}%
\left( \frac{a+b}{\sqrt{t}}\right) .  \label{Convolution_erfc}
\end{equation}
\end{theorem}

\begin{proof}
Apply the Laplace convolution theorem (\ref{Convolution_Laplace}), taking
into account the following inverse Laplace transforms \cite[Eqns.
2.2.2(15-16)]{Prudnikov5}:%
\begin{eqnarray*}
\mathcal{L}^{-1}\left[ \frac{\exp \left( -a\sqrt{s}\right) }{\sqrt{s}};t%
\right] &=&\frac{1}{\sqrt{\pi t}}\exp \left( -\frac{a^{2}}{4t}\right) ,\quad
\Re \left( a^{2}\right) >0 \\
\mathcal{L}^{-1}\left[ \frac{\exp \left( -a\sqrt{s}\right) }{s};t\right] &=&%
\mathrm{erfc}\left( \frac{a}{2\sqrt{t}}\right) ,\quad \Re \left(
a^{2}\right) >0,
\end{eqnarray*}%
to obtain%
\begin{equation*}
\frac{1}{\pi }\int_{0}^{t}\frac{\exp \left( -\frac{a^{2}}{4\tau }-\frac{b^{2}%
}{4\left( t-\tau \right) }\right) \,}{\sqrt{\tau \left( t-\tau \right) }}%
\,d\tau =\mathcal{L}^{-1}\left[ \frac{\exp \left( -\left( a+b\right) \sqrt{s}%
\right) }{s};t\right] =\mathrm{erfc}\left( \frac{a+b}{2\sqrt{t}}\right) .
\end{equation*}%
%
%
Perform the substitutions $a\mapsto 2a$ and $b\mapsto 2b$ to complete the
proof.
\end{proof}

\begin{example}
The integral (\ref{Convolution_erfc})\ appears in heat conduction in a
semi-infinite medium initially at zero temperature and with a prescribed
heat flux on the surface.

\begin{enumerate}
\item Physical setup \newline
Consider a semi-infinite medium of thermal diffusivity $\kappa $, and
thermal conductivity $K$, located at $x\geq 0$. Also, consider an outer
layer of thermal diffusivity $\kappa _{0}$ and thermal conductivity $K$,
located at $-L\leq x\leq 0$. Both media are initially at zero temperature,
and the outer layer is exposed to a temperature $T_{0}$ for $t>0$. According
to \cite[Sect. 12.8]{Jaeger}, if the thermal effusivities of both media are
equal, i.e.
\begin{equation*}
E_{0}=\frac{K_{0}}{\sqrt{\kappa _{0}}}=\frac{K}{\sqrt{\kappa }}=E,
\end{equation*}%
the heat flux at $x=0$ is given by
\begin{equation*}
q\left( t\right) =\frac{q_{0}}{\sqrt{t}}\exp \left( -\frac{a^{2}}{t}\right)
,\quad q_{0}=\frac{K_{0}T_{0}}{\sqrt{\pi \,\kappa _{0}}},\quad a=\frac{L}{2%
\sqrt{\kappa _{0}}}.
\end{equation*}%
Moreover, if we consider $t\ll \frac{L^{2}}{4\kappa _{0}}$, we have%
\begin{equation*}
q\left( t\right) \approx \frac{q_{s}}{\sqrt{t}}\exp \left( -\frac{a^{2}}{t}%
\right) ,\quad q_{s}=\frac{2EE_{0}T_{0}}{\sqrt{\pi }\left( E_{0}+E\right) }.
\end{equation*}

\item Governing equations (semi-infinite medium)

\begin{itemize}
\item Heat equation
\begin{equation}
\frac{\partial T\left( x,t\right) }{\partial t}=\kappa \frac{\partial
^{2}T\left( x,t\right) }{\partial x^{2}},\quad x>0,\,t>0.
\label{Heat_equation_1}
\end{equation}

\item Boundary conditions
\begin{equation}
-K\left. \frac{\partial T\left( x,t\right) }{\partial x}\right\vert
_{x=0}=q\left( t\right) ,  \label{BC_1}
\end{equation}%
\begin{equation}
\lim_{x\rightarrow \infty }T\left( x,t\right) =0.  \label{BC_2}
\end{equation}

\item Initial condition
\begin{equation}
T(x,0)=0.  \label{IC_1}
\end{equation}
\end{itemize}

\item As shown in \cite[Sect. 2.9]{Jaeger}, the solution of (\ref%
{Heat_equation_1})-(\ref{IC_1}) is given by the convolution integral:
\begin{equation*}
T\left( x,t\right) =\frac{\sqrt{\kappa }}{\sqrt{\pi }K}\int_{0}^{t}\frac{%
q\left( \tau \right) }{\sqrt{t-\tau }}\exp \left( -\frac{b^{2}}{t-\tau }%
\right) \,d\tau ,\quad b=\frac{x}{2\sqrt{\kappa }}.
\end{equation*}%
Therefore, applying the convolution integral given in (\ref{Convolution_erfc}%
), for the case $E=E_{0}$, we obtain:%
\begin{equation*}
T\left( x,t\right) =T_{0}\,\mathrm{erfc}\left( \frac{x}{2\sqrt{\kappa t}}+%
\frac{L}{2\sqrt{\kappa _{0}t}}\right) ,
\end{equation*}%
and for the case $t\ll \frac{L^{2}}{4\kappa _{0}}$, we have%
\begin{equation*}
T\left( x,t\right) \approx \frac{2E_{0}T_{0}}{E_{0}+E}\,\mathrm{erfc}\left(
\frac{x}{2\sqrt{\kappa t}}+\frac{L}{2\sqrt{\kappa _{0}t}}\right) .
\end{equation*}
\end{enumerate}
\end{example}

\begin{corollary}
Take $a=1$ and $b=0$ in (\ref{Convolution_erfc}), and perform the changes of
variables $z=t^{-1/2}$ and $\tau ^{-1}=z^{2}u$ to obtain the following
integral representation for $\left\vert \arg \left( z\right) \right\vert <%
\frac{\pi }{4}$:%
\begin{equation}
\mathrm{erfc}\left( z\right) =\frac{1}{\pi }\int_{1}^{\infty }\frac{\exp
\left( -z^{2}u\right) }{u\sqrt{u-1}}\,du.
\label{erfc_integral_representation}
\end{equation}
\end{corollary}

\begin{theorem}
For $\Re \left( a^{2}\right) ,\Re \left( b^{2}\right) >0$, and $n=1,2,\ldots
$ the following convolution integral holds true:%
\begin{eqnarray}
&&\int_{0}^{t}\frac{\exp \left( -\frac{a^{2}}{\tau }-\frac{b^{2}}{t-\tau }%
\right) \,}{\left[ \tau \left( t-\tau \right) \right] ^{n+1/2}}\,d\tau
\label{Convolution_erfc_n} \\
&=&2\sqrt{\pi }\,e^{-\left( a+b\right) ^{2}/t}\,\sum_{k=1}^{n}\sum_{\ell
=1}^{n}\frac{c_{n,k}\,c_{n,\ell }}{a^{2n-k}b^{2n-\ell }t^{\left( k+\ell
\right) /2}}\,H_{k+\ell -1}\left( \frac{a+b}{\sqrt{t}}\right)
\label{Proof_1} \\
&=&\frac{2\,e^{-\left( a^{2}+b^{2}\right) /t}}{t^{2n}}\,\sum_{k=0}^{2n-1}%
\binom{2n-1}{k}\left( \frac{a}{b}\right) ^{k-n+1/2}K_{k-n+1/2}\left( \frac{%
2ab}{t}\right) ,  \label{Proof_2}
\end{eqnarray}%
where%
\begin{equation}
c_{n,m}=\frac{\left( 2n-m-1\right) !}{2^{2n-k}\,\left( n-m\right) !\,\left(
m-1\right) !}.  \label{c_n,m_def}
\end{equation}
\end{theorem}

\begin{proof}
To prove (\ref{Proof_1}), define the differential operator:%
\begin{equation}
D_{a,b}^{n}=\left( \frac{1}{4ab}\frac{\partial ^{2}}{\partial a\,\partial b}%
\right) ^{n}=\left( \frac{1}{2a}\frac{\partial }{\partial a}\right)
^{n}\left( \frac{1}{2b}\frac{\partial }{\partial b}\right) ^{n}.
\label{D_a,b_def}
\end{equation}%
According to \cite{OEIS}, and performing the change $\ell =n-k$, we obtain%
\begin{eqnarray}
\left( \frac{1}{2x}\frac{\partial }{\partial x}\right) ^{n} &=&\frac{1}{2^{n}%
}\sum_{k=0}^{n-1}\frac{\left( n-1+k\right) !}{2^{k}\,\left( n-1-k\right)
!\,k!}\frac{\left( -1\right) ^{k}}{x^{n+k}}\left( \frac{\partial }{\partial x%
}\right) ^{n-k}  \label{(1/2x Dx)^n} \\
&=&\sum_{\ell =1}^{n}\frac{\left( 2n-\ell -1\right) !}{2^{2n-\ell }\left(
\ell -1\right) !\,\left( n-\ell \right) !}\frac{\left( -1\right) ^{n-\ell }}{%
x^{2n-\ell }}\left( \frac{\partial }{\partial x}\right) ^{\ell }.  \notag
\end{eqnarray}%
Therefore,
\begin{eqnarray}
D_{a,b}^{n} &=&\sum_{k=1}^{n}\sum_{\ell =1}^{n}\frac{\left( -1\right)
^{k+n}\,c_{n,k}\,c_{n,\ell }}{a^{2n-k}\,b^{2n-\ell }}\left( \frac{\partial }{%
\partial a}\right) ^{k}\left( \frac{\partial }{\partial b}\right) ^{\ell }
\label{D_a,b_LEFT} \\
&=&\sum_{k=1}^{n}\sum_{\ell =1}^{n}\frac{\left( -1\right)
^{k+n}\,c_{n,k}\,c_{n,\ell }}{a^{2n-k}\,b^{2n-\ell }}\left( \frac{\partial }{%
\partial \left( a+b\right) }\right) ^{k+\ell }.  \notag
\end{eqnarray}%
where the combinatorial coefficients $c_{n,m}$ are given by (\ref{c_n,m_def}%
). On the one hand, applying (\ref{D_a,b_def}), we have%
\begin{equation}
D_{a,b}^{n}\left[ \frac{\exp \left( -\frac{a^{2}}{\tau }-\frac{b^{2}}{t-\tau
}\right) \,}{\sqrt{\tau \left( t-\tau \right) }}\right] =\frac{\exp \left( -%
\frac{a^{2}}{\tau }-\frac{b^{2}}{t-\tau }\right) \,}{\left[ \tau \left(
t-\tau \right) \right] ^{n+1/2}}.  \label{D^n_a,b_LEFT}
\end{equation}%
On the other hand, applying (\ref{D_a,b_LEFT}), and taking into account (\ref%
{erfc_def}) and (\ref{D^n[erf]}), we obtain%
\begin{eqnarray}
&&D_{a,b}^{n}\left[ \pi \,\mathrm{erfc}\left( \frac{a+b}{\sqrt{t}}\right) %
\right]  \label{D^n_a,b_RIGHT} \\
&=&-\pi \,D_{a,b}^{n}\left[ \,\mathrm{erf}\left( \frac{a+b}{\sqrt{t}}\right) %
\right]  \notag \\
&=&2\sqrt{\pi }\,e^{-\left( a+b\right) ^{2}/t}\,\sum_{k=1}^{n}\sum_{\ell
=1}^{n}\frac{c_{n,k}\,c_{n,\ell }}{a^{2n-k}b^{2n-\ell }t^{\left( k+\ell
\right) /2}}\,H_{k+\ell -1}\left( \frac{a+b}{\sqrt{t}}\right) .  \notag
\end{eqnarray}%
Finally, apply the differential operator $D_{a,b}^{n}$ to both sides of (\ref%
{Convolution_erfc}), and take into account (\ref{D^n_a,b_LEFT})\ and (\ref%
{D^n_a,b_RIGHT})\ to complete the first part of the proof, i.e. (\ref%
{Proof_1}).

For the second part of the proof, i.e. (\ref{Proof_2}), perform the change
of variables $u=\frac{\tau }{t-\tau }$ in (\ref{Convolution_erfc_n})\ to
obtain%
\begin{eqnarray*}
&&\int_{0}^{t}\frac{\exp \left( -\frac{a^{2}}{\tau }-\frac{b^{2}}{t-\tau }%
\right) \,}{\left[ \tau \left( t-\tau \right) \right] ^{n+1/2}}\,d\tau \\
&=&\frac{e^{-\left( a^{2}+b^{2}\right) /t}}{t^{2n}}\,\int_{0}^{\infty
}u^{-n-1/2}\,\left( 1+u\right) ^{2n-1}\exp \left( -\frac{b^{2}+a^{2}/u}{t}%
\right) \,du.
\end{eqnarray*}%
Now, according to the binomial theorem, we have
\begin{equation*}
\left( 1+u\right) ^{2n-1}=\sum_{k=0}^{2n-1}\binom{2n-1}{k}u^{k}.
\end{equation*}%
Thus%
\begin{eqnarray*}
&&\int_{0}^{t}\frac{\exp \left( -\frac{a^{2}}{\tau }-\frac{b^{2}}{t-\tau }%
\right) \,}{\left[ \tau \left( t-\tau \right) \right] ^{n+1/2}}\,d\tau \\
&=&\frac{e^{-\left( a^{2}+b^{2}\right) /t}}{t^{2n}}\,\sum_{k=0}^{2n-1}\binom{%
2n-1}{k}\int_{0}^{\infty }u^{k-n-1/2}\,\exp \left( -\frac{b^{2}+a^{2}/u}{t}%
\right) \,du.
\end{eqnarray*}%
Finally, apply the following known integral \cite[Eqn. 2.3.16(1)]{Prudnikov1}%
:
\begin{equation*}
\int_{0}^{\infty }u^{\alpha -1}\,\exp \left( -pu-\frac{q}{u}\right)
\,du=2\left( \frac{p}{q}\right) ^{\alpha /2}K_{\alpha }\left( 2\sqrt{p\,q}%
\right) ,
\end{equation*}%
to complete the second part of the proof.
\end{proof}

\begin{remark}
Straightforwardly from (\ref{Proof_1}), the convolution integral (\ref%
{Convolution_erfc_n})\ is expressed in terms of elementary functions. This
is also clear from (\ref{Proof_2})\ taking into account the property (\ref%
{K_nu=K_-nu})\ and the following expansion of the Macdonald function for
half-integer order \cite[Eqn. 8.468]{Gradshteyn}:%
\begin{equation*}
K_{n+1/2}\left( z\right) =\sqrt{\frac{\pi }{2z}}e^{-z}\sum_{k=0}^{n}\frac{%
\left( n+k\right) !}{k!\left( n-k\right) !\,\left( 2z\right) ^{k}}.
\end{equation*}%
For instance, for $n=1$, (\ref{Convolution_erfc_n})\ reduces to%
\begin{equation}
\int_{0}^{t}\frac{\exp \left( -\frac{a^{2}}{\tau }-\frac{b^{2}}{t-\tau }%
\right) }{\left[ \tau \left( t-\tau \right) \right] ^{3/2}}\,\,d\tau =\frac{%
\sqrt{\pi }\left( a+b\right) }{a\,b\,t^{3/2}}\exp \left( -\frac{\left(
a+b\right) ^{2}}{t}\right) ,  \label{Convolution_erfc_n=1}
\end{equation}%
and for $n=2$, we have%
\begin{eqnarray}
&&\int_{0}^{t}\frac{\exp \left( -\frac{a^{2}}{\tau }-\frac{b^{2}}{t-\tau }%
\right) }{\left[ \tau \left( t-\tau \right) \right] ^{5/2}}\,\,d\tau
\label{Convolution_erfc_n=2} \\
&=&\frac{\sqrt{\pi }\left[ 2ab\left( a+b\right) ^{3}+\left(
a^{3}+b^{3}\right) t\right] }{2a^{3}b^{3}\,t^{7/2}}\exp \left( -\frac{\left(
a+b\right) ^{2}}{t}\right) .  \notag
\end{eqnarray}
\end{remark}

\begin{example}
The integral (\ref{Convolution_erfc_n=1}) appears in heat conduction in a
semi-infinite medium initially at zero temperature, with a prescribed
temperature evolution on its surface.

\begin{enumerate}
\item Physical setup\newline
Consider a semi-infinite rod located along $x\geq 0$ and immersed in an
insulating material. A localized laser pulse discharges a sudden, massive
burst of energy at a single point deep within the insulating material at
time $t=0$ and at distance $r_{0}$ from the surface of the rod, which is
located at $x=0$. The temperature diffuses outward from the point source
spherically. At a distance $r$ from the point source, the temperature
evolves according to \cite[Sect. 10.2]{Jaeger}:%
\begin{equation*}
T_{\text{source}}\left( r,t\right) =\frac{Q}{\rho \,c_{p}\left( 4\pi
\,\kappa _{0}t\right) ^{3/2}}\exp \left( -\frac{r^{2}}{4\kappa _{0}t}\right)
,
\end{equation*}%
where $Q$ is the energy released, $\rho $ is the density of the insulating
material, $c_{p}$ is its specific heat capacity, and $\kappa _{0}$ is the
thermal diffusivity of the insulating material. Therefore, the temperature
on the surface of the rod evolves as $T\left( 0,t\right) =T_{\text{source}%
}\left( r_{0},t\right) $.

\item Governing equations (in semi-infinite medium)

\begin{itemize}
\item Heat equation
\begin{equation}
\frac{\partial T\left( x,t\right) }{\partial t}=\kappa \frac{\partial
^{2}T\left( x,t\right) }{\partial x^{2}},\quad x>0,\,t>0.
\label{Heat_Equation_source}
\end{equation}

\item Boundary conditions%
\begin{equation}
T\left( 0,t\right) =\frac{q}{t^{3/2}}\exp \left( -\frac{a^{2}}{t}\right)
,\quad q=\frac{Q}{\rho \,c_{p}\left( 4\pi \kappa _{0}\right) ^{3/2}},\quad a=%
\frac{r_{0}}{2\sqrt{\kappa _{0}}}  \label{BC_source_0}
\end{equation}%
\begin{equation}
\lim_{x\rightarrow \infty }T\left( x,t\right) =0.  \label{BC_source_infinity}
\end{equation}

\item Initial condition
\begin{equation}
T(x,0)=0.  \label{IC_source}
\end{equation}
\end{itemize}

\item According to \cite[Sect. 2.5]{Jaeger}, the solution of the boundary
value problem (\ref{Heat_Equation_source})-(\ref{IC_source}) is given by the
convolution integral:
\begin{eqnarray*}
T\left( x,t\right) &=&\frac{b}{\sqrt{\pi }}\int_{0}^{t}\frac{T\left( 0,\tau
\right) }{\left( t-\tau \right) ^{3/2}}\exp \left( -\frac{b^{2}}{t-\tau }%
\right) \,d\tau ,\quad b=\frac{x}{2\sqrt{\kappa }} \\
&=&\frac{q\,b}{\sqrt{\pi }}\int_{0}^{t}\frac{\exp \left( -\frac{a^{2}}{\tau }%
-\frac{b^{2}}{t-\tau }\right) }{\left[ \tau \left( t-\tau \right) \right]
^{3/2}}\,\,d\tau .
\end{eqnarray*}%
Therefore, applying the convolution integral given in (\ref%
{Convolution_erfc_n=1}), we arrive at%
\begin{equation*}
T\left( x,t\right) =\frac{Q}{\rho \,c_{p}\left( 4\pi \,\kappa _{0}t\right)
^{3/2}}\left( 1+\frac{x}{r_{0}}\sqrt{\frac{\kappa _{0}}{\kappa }}\right)
\exp \left( -\frac{1}{4t}\left[ \frac{x}{\sqrt{\kappa }}+\frac{r_{0}}{\sqrt{%
\kappa _{0}}}\right] ^{2}\right) .
\end{equation*}
\end{enumerate}
\end{example}

\begin{theorem}
\label{Theorem: Laplace 2} For $\Re \left( \nu \right) >0$ and $a\in
\mathbb{C}
$, the following convolution integral holds true:%
\begin{equation}
\int_{0}^{t}\frac{1-\exp \left( -a\,\tau \right) }{\tau }\left( t-\tau
\right) ^{\nu }d\tau =\frac{a\,t^{\nu +1}}{\nu +1}\,_{2}F_{2}\left(
\begin{array}{c}
1,1 \\
2,\nu +2%
\end{array}%
;-a\,t\right) .  \label{Convolution_(1-exp)}
\end{equation}
\end{theorem}

\begin{proof}
Apply the Laplace convolution theorem (\ref{Convolution_Laplace}), taking
into account the inverse Laplace transforms \cite[Eqns.
2.5.3.(1)-(2)\&2.1.1(1)]{Prudnikov5}:%
\begin{eqnarray*}
\mathcal{L}^{-1}\left[ \log \left( 1+\frac{a}{s}\right) ;t\right] &=&\frac{%
1-\exp \left( -a\,t\right) }{t}, \\
\mathcal{L}^{-1}\left[ \frac{1}{s^{\nu }};t\right] &=&\frac{t^{\nu -1}}{%
\Gamma \left( \nu \right) }, \\
\mathcal{L}^{-1}\left[ \frac{1}{s^{\nu }}\log \left( 1+\frac{a}{s}\right) ;t%
\right] &=&\frac{a\,t^{\nu }}{\Gamma \left( \nu +1\right) }\,_{2}F_{2}\left(
\begin{array}{c}
1,1 \\
2,\nu +1%
\end{array}%
;-a\,t\right) .
\end{eqnarray*}%
Therefore,
\begin{eqnarray*}
\frac{1}{\Gamma \left( \nu \right) }\int_{0}^{t}\,\frac{1-\exp \left(
-a\,t\right) }{t}\,\left( t-\tau \right) ^{\nu -1}d\tau &=&\mathcal{L}^{-1}%
\left[ \frac{1}{s^{\nu }}\log \left( 1+\frac{a}{s}\right) ;t\right] \\
&=&\frac{a\,t^{\nu }}{\Gamma \left( \nu +1\right) }\,_{2}F_{2}\left(
\begin{array}{c}
1,1 \\
2,\nu +1%
\end{array}%
;-a\,t\right) .
\end{eqnarray*}%
Apply the property (\ref{Gamma_factorial})\ and perform the substitution $%
\nu \mapsto \nu +1$ to complete the proof.

Alternatively, we can derive (\ref{Convolution_(1-exp)})\ as follows. Expand
the exponential function to obtain
\begin{equation*}
\frac{1-\exp \left( -a\,\tau \right) }{\tau }=\sum_{n=0}^{\infty }\frac{%
\left( -a\tau \right) ^{n}}{\left( n+1\right) !},
\end{equation*}%
thus%
\begin{equation}
\int_{0}^{t}\frac{1-\exp \left( -a\,\tau \right) }{\tau }\left( t-\tau
\right) ^{\nu }d\tau =\sum_{n=0}^{\infty }\frac{\left( -a\right) ^{n}}{%
\left( n+1\right) !}\underset{I}{\underbrace{\int_{0}^{t}\tau ^{n}\left(
t-\tau \right) ^{\nu }d\tau }}.  \label{I_definition}
\end{equation}%
Note that, performing the substitution $\tau \mapsto t-\tau $ and taking
into account the binomial theorem, we have%
\begin{eqnarray}
I &=&\int_{0}^{t}\tau ^{\nu }\left( t-\tau \right) ^{n}d\tau  \label{J_def}
\\
&=&\sum_{k=0}^{n}\binom{n}{k}\left( -1\right) ^{k}t^{n-1}\int_{0}^{t}\tau
^{\nu +k}\,d\tau =t^{n+\nu +1}\sum_{k=0}^{n}\frac{\left( -1\right) ^{k}\,n!}{%
k!\left( n-k\right) !\left( \nu +k+1\right) }.  \notag
\end{eqnarray}%
However, from the Pochhammer symbol properties (\ref{Pochhammer_1}) and (\ref%
{Pochhammer_2}), we have%
\begin{eqnarray}
\left( -n\right) _{k} &=&\frac{\left( -1\right) ^{k}\,n!}{\left( n-k\right) !%
},  \label{(-n)_k} \\
\frac{1}{\nu +k+1} &=&\frac{\left( \nu +1\right) _{k}}{\left( \nu +1\right)
\,\left( \nu +2\right) _{k}},  \label{1/(k+a)=Pochhammer}
\end{eqnarray}%
thus, taking into account the definition of the generalized hypergeometric
function (\ref{pFq_def}) and the Chu-Vandermonde summation formula (\ref%
{Chu_formula}), we have%
\begin{eqnarray}
I &=&\frac{t^{n+\nu +1}}{\nu +1}\sum_{k=0}^{n}\frac{\left( -n\right)
_{k}\,\left( \nu +1\right) _{k}}{k!\,\left( \nu +2\right) _{k}}=\frac{%
t^{n+\nu +1}}{\nu +1}\,_{2}F_{1}\left(
\begin{array}{c}
-n,\nu +2 \\
\nu +1%
\end{array}%
;1\right)  \label{J_resultado} \\
&=&\frac{t^{n+\nu +1}}{\nu +1}\frac{\left( 1\right) _{n}}{\left( \nu
+2\right) _{n}}.  \notag
\end{eqnarray}%
Substitute back (\ref{J_resultado})\ into (\ref{I_definition})\ and take
into account again the definition of the generalized hypergeometric function
(\ref{pFq_def}) to complete the proof.
\end{proof}

\begin{example}
The integral (\ref{Convolution_(1-exp)}) arises in the context of
viscoelastic materials following the Becker model and subjected to a power
law stress.

\begin{enumerate}
\item Physical setup \newline
Consider a linear viscoelastic material subjected to a tensile or
compressive stress $\sigma \left( t\right) $. The strain produced in the
material is given by \cite[Eqn. 2.6a]{MainardiBook}:
\begin{equation}
\varepsilon \left( t\right) =\sigma \left( t\right) \,J\left( 0^{+}\right)
+\int_{0}^{t}\dot{J}\left( t-\tau \right) \,\sigma \left( \tau \right)
\,d\tau ,  \label{strain_convolution}
\end{equation}%
where $J\left( t\right) $ is the creep compliance function (i.e. the strain
response to a unit step of stress).

\item Constitutive model \newline
Consider a material following the creep compliance of the Becker model \cite[%
Eqn. 2.37]{MainardiBook}:%
\begin{equation}
J\left( t\right) =a\,\mathrm{Ein}\left( \frac{t}{t_{0}}\right) ,\quad a>0,\
t_{0}>0,  \label{Creep_Becker}
\end{equation}%
thus, according to (\ref{Ein_series}),
\begin{equation}
J\left( 0^{+}\right) =0,  \label{J(0+)_Becker}
\end{equation}%
and, according to (\ref{Ein_def}), the creep rate is%
\begin{equation}
\dot{J}\left( t\right) =a\,\frac{1-\exp \left( -t/t_{0}\right) }{t}.
\label{Creep_rate_Becker}
\end{equation}%
Also, consider a viscoelastic material subjected to a stress $\sigma \left(
t\right) $ that varies with time following a power law,%
\begin{equation}
\sigma \left( t\right) =\sigma _{0}\,\left( \frac{t}{t_{\sigma }}\right)
^{\nu },\quad \sigma _{0}>0,\ t_{\sigma }>0.  \label{sigma_power_law}
\end{equation}

\item Substitute (\ref{J(0+)_Becker})-(\ref{sigma_power_law}) into (\ref%
{strain_convolution}), taking into account the integral calculated in (\ref%
{Convolution_(1-exp)}), to obtain the strain in the material described
above:
\begin{eqnarray*}
\varepsilon \left( t\right)  &=&\int_{0}^{t}\dot{J}\left( \tau \right)
\,\sigma \left( t-\tau \right) \,d\tau  \\
&=&a\frac{\,\sigma _{0}}{t_{\sigma }^{\nu }}\int_{0}^{t}\frac{1-\exp \left(
-\,\tau /t_{0}\right) }{\tau }\,\left( t-\tau \right) ^{\nu }d\tau  \\
&=&a\,\sigma _{0}\frac{t}{t_{0}}\left( \frac{t}{t_{\sigma }}\right) ^{\nu
}\,_{2}F_{2}\left(
\begin{array}{c}
1,1 \\
2,\nu +2%
\end{array}%
;-\,\frac{t}{t_{0}}\right) .
\end{eqnarray*}
\end{enumerate}
\end{example}

\subsection{Parabolic cylinder function}

\begin{theorem}
\label{Theorem: Laplace 3} For $\Re \left( a^{2}\right) >0$ and $\Re \left(
b^{2}\right) >0$, the following convolution integral holds:
\begin{eqnarray}
&&\int_{0}^{t}\exp \left( -\frac{1}{4}\left[ \frac{a^{2}}{\tau }+\frac{b^{2}%
}{t-\tau }\right] \right) \frac{\,D_{2\nu +1}\left( \frac{a}{\sqrt{\tau }}%
\right) \,D_{2\mu +1}\left( \frac{b}{\sqrt{t-\tau }}\right) }{\tau ^{\nu
+1}\left( t-\tau \right) ^{\mu +1}}\,\,d\tau   \label{Int_parabolic_general}
\\
&=&\frac{\sqrt{2\pi }}{t^{\mu +\nu +1}}\exp \left( -\frac{\left( a+b\right)
^{2}}{4t}\right) \,D_{2\left( \mu +\nu \right) +1}\left( \frac{a+b}{\sqrt{t}}%
\right) .  \notag
\end{eqnarray}
\end{theorem}

\begin{proof}
Apply the Laplace convolution theorem (\ref{Convolution_Laplace}), taking
into account the following inverse Laplace transform (see \cite[Eqn.
2.2.2(10)]{Prudnikov5}):
\begin{equation*}
\mathcal{L}^{-1}\left[ s^{\beta }\exp \left( -\alpha \sqrt{s}\right) ;t%
\right] =\frac{t^{-\beta -1}}{2^{\beta +1/2}\sqrt{\pi }}\exp \left( -\frac{%
\alpha ^{2}}{8t}\right) \,D_{2\beta +1}\left( \frac{\alpha }{\sqrt{2t}}%
\right) ,\quad \Re \left( \alpha ^{2}\right) >0,
\end{equation*}
to obtain%
\begin{eqnarray*}
&&\frac{1}{2^{\mu +\nu +1}\pi }\int_{0}^{t}\exp \left( -\frac{1}{8}\left[
\frac{a^{2}}{\tau }+\frac{b^{2}}{t-\tau }\right] \right) \,\frac{D_{2\nu
+1}\left( \frac{a}{\sqrt{2\tau }}\right) \,D_{2\mu +1}\left( \frac{b}{\sqrt{%
2\left( t-\tau \right) }}\right) }{\tau ^{\nu +1}\left( t-\tau \right) ^{\mu
+1}}\,\,d\tau \\
&=&\mathcal{L}^{-1}\left[ s^{\mu +\nu }\exp \left( -\left( a+b\right) \sqrt{s%
}\right) ;t\right] \\
&=&\frac{1}{2^{\mu +\nu +1/2}\sqrt{\pi }\,t^{\mu +\nu +1}}\exp \left( -\frac{%
\left( a+b\right) ^{2}}{8t}\right) \,D_{2\left( \mu +\nu \right) +1}\left(
\frac{a}{\sqrt{2t}}\right) .
\end{eqnarray*}%
Perform the substitutions $a\mapsto \sqrt{2}a$, $b\mapsto \sqrt{2}b$, and
simplify the result to complete the proof.
\end{proof}

\begin{corollary}
\label{Corollary a Laplace 3} For $\Re \left( a^{2}\right) >0$ and $\Re
\left( b^{2}\right) >0$, the following convolution integral holds:
\begin{eqnarray}
&&\int_{0}^{t}\exp \left( -\frac{a^{2}}{4\tau }\right) \,D_{2\nu +1}\left(
\frac{a}{\sqrt{\tau }}\right) \,\mathrm{erfc}\left( \frac{b}{\sqrt{t-\tau }}%
\right) \,\frac{\,d\tau }{\tau ^{\nu +1}}  \label{Int_parabolic_erfc} \\
&=&\frac{2}{t^{\nu }}\exp \left( -\frac{\left( a+\sqrt{2}b\right) ^{2}}{4t}%
\right) \,D_{2\nu -1}\left( \frac{a+\sqrt{2}b}{\sqrt{t}}\right) .  \notag
\end{eqnarray}
\end{corollary}

\begin{proof}
Take $\mu =-1$ in (\ref{Int_parabolic_general}) to obtain
\begin{eqnarray}
&&\int_{0}^{t}\exp \left( -\frac{1}{4}\left[ \frac{a^{2}}{\tau }+\frac{b^{2}%
}{t-\tau }\right] \right) \,D_{2\nu +1}\left( \frac{a}{\sqrt{\tau }}\right)
\,D_{-1}\left( \frac{b}{\sqrt{t-\tau }}\right) \,\frac{\,d\tau }{\tau ^{\nu
+1}}  \label{Int_parabolic_D_-1} \\
&=&\frac{\sqrt{2\pi }}{t^{\nu }}\exp \left( -\frac{\left( a+b\right) ^{2}}{4t%
}\right) \,D_{2\nu -1}\left( \frac{a+b}{\sqrt{t}}\right) .  \notag
\end{eqnarray}%
Finally, insert (\ref{D_-1_resultado}) into (\ref{Int_parabolic_D_-1}) to
complete the proof.
\end{proof}

\begin{corollary}
\label{Corollary b Laplace 3} For $\Re \left( a^{2}\right)>0$ and $\Re
\left( b^{2}\right) >0$, the following convolution integral holds true:%
\begin{equation}
\int_{0}^{t}\exp \left( -\frac{a^{2}}{t-\tau }\right) \,\,\mathrm{erfc}%
\left( \frac{b}{\sqrt{\tau }}\right) \,\frac{\,d\tau }{\sqrt{t-\tau }}=2%
\sqrt{\pi t}\,\,\mathrm{ierfc}\left( \frac{a+b}{\sqrt{t}}\right) .
\label{Convolution_ierfc}
\end{equation}
\end{corollary}

\begin{proof}
Perform the substitutions $\tau \mapsto t-\tau $ and $\nu =-\frac{1}{2}$ in (%
\ref{Int_parabolic_erfc}), to obtain
\begin{eqnarray*}
&&\int_{0}^{t}\exp \left( -\frac{a^{2}}{4\left( t-\tau \right) }\right)
\,D_{0}\left( \frac{a}{\sqrt{t-\tau }}\right) \,\mathrm{erfc}\left( \frac{b}{%
\sqrt{\tau }}\right) \,\frac{\,d\tau }{\sqrt{t-\tau }} \\
&=&2\sqrt{t}\exp \left( -\frac{\left( a+\sqrt{2}b\right) ^{2}}{4t}\right)
\,D_{-2}\left( \frac{a+\sqrt{2}b}{\sqrt{t}}\right) .
\end{eqnarray*}%
Finally, apply the reduction formulas (\ref{D_0_reduction}),\ and (\ref%
{D_-n-1_reduction})\ for $n=1$ and perform the substitution $a\mapsto \sqrt{2%
}a$ to complete the proof.
\end{proof}

\begin{example}
The integral given in (\ref{Convolution_ierfc}) appears in the heat
conduction in a semi-infinite medium, initially at zero temperature, with a
prescribed heat flux on its surface.

\begin{enumerate}
\item Physical setup \newline
Consider a semi-infinite medium of thermal diffusivity $\kappa $, and
thermal conductivity $K$, located at $x\geq 0$. Also, consider an outer
layer of thermal diffusivity $\kappa _{0}$ and thermal conductivity $K_{0}$,
located at $-L\leq x\leq 0$. Both media are initially at zero temperature,
and the outer layer is exposed to a heat flux $q$ for $t>0$. According to
\cite[Chap. 6]{Ozisik}, if the thermal effusivities of both mediums are
equal, i.e.,
\begin{equation*}
E_{0}=\frac{K_{0}}{\sqrt{\kappa _{0}}}=\frac{K}{\sqrt{\kappa }}=E,
\end{equation*}%
the heat flux at $x=0$ is given by
\begin{equation}
q\left( t\right) =q\,\mathrm{erfc}\left( \frac{b}{\sqrt{t}}\right) ,\quad b=%
\frac{L}{2\sqrt{\kappa _{0}}}.  \label{q(t)_erfc}
\end{equation}

\item Governing equations (semi-infinite medium)

\begin{itemize}
\item Heat equation
\begin{equation}
\frac{\partial T\left( x,t\right) }{\partial t}=\kappa \frac{\partial
^{2}T\left( x,t\right) }{\partial x^{2}},\quad x>0,\,t>0.
\label{Heat_Equation_q}
\end{equation}

\item Boundary conditions
\begin{equation}
-K\left. \frac{\partial T\left( x,t\right) }{\partial x}\right\vert
_{x=0}=q\left( t\right) ,  \label{BC_q}
\end{equation}%
\begin{equation}
\lim_{x\rightarrow \infty }T\left( x,t\right) =0.  \label{BC_infinity_q}
\end{equation}

\item Initial condition
\begin{equation}
T(x,0)=0.  \label{IC_q}
\end{equation}
\end{itemize}

\item According to \cite[Sect. 2.9]{Jaeger}, the solution of (\ref%
{Heat_Equation_q})-(\ref{IC_q}) is given by the convolution integral:
\begin{equation*}
T\left( x,t\right) =\frac{\sqrt{\kappa }}{\sqrt{\pi }K}\int_{0}^{t}\frac{%
q\left( \tau \right) }{\sqrt{t-\tau }}\exp \left( -\frac{a^{2}}{t-\tau }%
\right) \,d\tau ,\quad a=\frac{x}{2\sqrt{\kappa }}.
\end{equation*}%
Therefore, inserting the heat flux $q\left( t\right) $ given in (\ref%
{q(t)_erfc}),\ and applying the convolution integral (\ref{Convolution_ierfc}%
), we obtain:%
\begin{equation*}
T\left( x,t\right) =\frac{2q}{E}\sqrt{t}\,\mathrm{ierfc}\left( \frac{x}{2%
\sqrt{\kappa t}}+\frac{L}{2\sqrt{\kappa _{0}t}}\right) .
\end{equation*}
\end{enumerate}
\end{example}

\begin{theorem}
For $\Re \left( a^{2}\right) >0$, the following convolution integral holds:
\begin{eqnarray}
&&\int_{0}^{t}\exp \left( -\frac{a^{2}}{4\tau }\right) \,D_{2\nu +1}\left(
\frac{a}{\sqrt{\tau }}\right) \,\frac{\left( t-\tau \right) ^{\mu -1}}{\tau
^{\nu +1}}\,d\tau  \label{Int_repeated_integration} \\
&=&\frac{2^{\mu }\,\Gamma \left( \mu \right) }{t^{\nu -\mu +1}}\exp \left( -%
\frac{a^{2}}{4t}\right) \,D_{2\left( \nu -\mu \right) +1}\left( \frac{a}{%
\sqrt{t}}\right) .  \notag
\end{eqnarray}
\end{theorem}

\begin{proof}
Take $b=0$ in (\ref{Int_parabolic_erfc}), and apply (\ref{erfc(0)}) to
obtain (see \cite[Eqn. 1.7.1(9)]{Prudnikov2}):
\begin{equation}
\int_{0}^{t}\exp \left( -\frac{a^{2}}{4\tau }\right) \,D_{2\nu +1}\left(
\frac{a}{\sqrt{\tau }}\right) \,\frac{\,d\tau }{\tau ^{\nu +1}}=\frac{2}{%
t^{\nu }}\exp \left( -\frac{a^{2}}{4t}\right) \,D_{2\nu -1}\left( \frac{a}{%
\sqrt{t}}\right) .  \label{Int_parabolic_particular}
\end{equation}%
Apply the repeated integration formula (\ref{Repeated_integration}) to (\ref%
{Int_parabolic_particular}) to complete the proof.
\end{proof}

\begin{corollary}
Take $\nu =0$ in (\ref{Int_repeated_integration}), and perform the
substitutions $a\mapsto \sqrt{2}a$, and $\mu \rightarrow \mu +1$, taking
into account the reduction formula (\ref{D_1_reduction}), to obtain for $\Re
\left( a^{2}\right) >0$,
\begin{eqnarray}
&&\int_{0}^{t}\exp \left( -\frac{a^{2}}{t-\tau }\right) \frac{\tau ^{\mu }}{%
\left( t-\tau \right) ^{3/2}}\,d\tau  \label{Convolution_power_law} \\
&=&\frac{2^{\mu +1/2}\,\Gamma \left( \mu +1\right) }{a}\,t^{\mu }\,\exp
\left( -\frac{a^{2}}{2t}\right) \,D_{-1-2\mu }\left( \frac{\sqrt{2}a}{\sqrt{t%
}}\right) .  \notag
\end{eqnarray}
\end{corollary}

\begin{proof}
As an alternative proof of (\ref{Convolution_power_law}), proceed as
follows. Perform the changes of variables $\tau \mapsto t-\tau $ and $\tau
=t\,u$, to obtain
\begin{equation*}
I=\int_{0}^{t}\exp \left( -\frac{a^{2}}{t-\tau }\right) \frac{\tau ^{\mu }}{%
\left( t-\tau \right) ^{3/2}}\,d\tau =t^{\mu -1/2}\int_{0}^{1}\exp \left( -%
\frac{a^{2}}{ut}\right) \left( 1-u\right) ^{\mu }\,u^{3/2}\,du.
\end{equation*}%
Further, perform the change of variables $x=\frac{1-u}{u}$, and apply the
integral representation of the Tricomi function (\ref{Tricomi_integral}), to
obtain%
\begin{eqnarray*}
I &=&e^{-b}t^{\mu -1/2}\int_{0}^{\infty }x^{\mu }\left( 1+x\right) ^{-\mu
-1/2}\exp \left( -\frac{a^{2}}{t}\,x\right) \,dx \\
&=&\exp \left( -\frac{a^{2}}{t}\right) \,t^{\mu -1/2}\,\mathrm{U}\left( \mu
+1,\frac{3}{2},\frac{a^{2}}{t}\right) .
\end{eqnarray*}%
Finally, apply the reduction formula (\ref{Tricomi_reduction}) to complete
the proof.
\end{proof}

\begin{example}
The integral (\ref{Convolution_power_law}) appears in the heat conduction in
a semi-infinite medium initially at zero temperature, with a power law heat
flux on its surface.

\begin{enumerate}
\item Physical setup \newline
Consider a semi-infinite solid ($x\geq 0$) with thermal diffusivity\ $\kappa
$, and thermal conductivity $K$, that is initially at zero temperature.
Also, at $x=0$ heat is supplied through a contact region, in such a way that
the region through which heat is delivered expands with time. Therefore, if $%
q_{0}$ is the local heat flux density ($\mathrm{W}\,\mathrm{m}^{-2}$) and $%
A\left( t\right) $ is the effective contact area, then the heat flux $%
q\left( t\right) $ at the boundary $x=0$ evolves as%
\begin{equation*}
q\left( t\right) =q_{0}\,A\left( t\right) .
\end{equation*}%
In many physical spreading processes, the contact radius $R\left( t\right) $
follows a similarity law, i.e. $R\left( t\right) \propto t^{\beta }$. For
instance, according to Tanner's law \cite{Tanner}, $R\left( t\right) \sim
t^{1/10}$. Therefore, if $A\left( t\right) $ is a disk-shaped region, we
have
\begin{equation*}
A\left( t\right) =\pi \,R\left( t\right) ^{2}\propto \,t^{2\beta },
\end{equation*}%
thus the heat flux at $x=0$ follows a power law function,
\begin{equation*}
q\left( t\right) =C\,t^{\mu },
\end{equation*}%
where $C$ and $\mu $ are given constants.

\item Governing Equations

\begin{itemize}
\item Heat equation
\begin{equation}
\frac{\partial T\left( x,t\right) }{\partial t}=\kappa \frac{\partial
^{2}T\left( x,t\right) }{\partial x^{2}},\quad x>0,\,t>0.
\label{Heat_Equation_power_law}
\end{equation}

\item Boundary conditions
\begin{equation}
-K\left. \frac{\partial T\left( x,t\right) }{\partial x}\right\vert
_{x=0}=q\left( t\right) ,  \label{BC_power_law}
\end{equation}%
\begin{equation}
\lim_{x\rightarrow \infty }T\left( x,t\right) =0.
\label{BC_at_infinity_power_law}
\end{equation}

\item Initial condition
\begin{equation}
T(x,0)=0.  \label{IC_power_law}
\end{equation}
\end{itemize}

\item According to \cite[Sect. 2.9]{Jaeger}, the solution of (\ref%
{Heat_Equation_power_law})-(\ref{IC_power_law}) is given by the convolution
integral:
\begin{equation*}
T\left( x,t\right) =\frac{\sqrt{\kappa }}{K\sqrt{\pi }}\int_{0}^{t}q\left(
t-\tau \right) \exp \left( -\frac{x^{2}}{4\kappa \tau }\right) \,\frac{d\tau
}{\sqrt{\tau }},
\end{equation*}%
thus the heat flux at an arbitrary point $x\geq 0$ is%
\begin{eqnarray*}
q\left( x,t\right) &=&-K\,\frac{\partial T\left( x,t\right) }{\partial x} \\
&=&\frac{x}{2\sqrt{\pi \,\kappa }}\int_{0}^{t}q\left( \tau \right) \exp
\left( -\frac{x^{2}}{4\kappa \left( t-\tau \right) }\right) \,\frac{d\tau }{%
\left( t-\tau \right) ^{3/2}}.
\end{eqnarray*}%
Take $a=\frac{x}{2\sqrt{\kappa }}$, $q\left( t\right) =C\,t^{\mu }$, and
apply the integral given in (\ref{Convolution_power_law}), to obtain
\begin{equation*}
q\left( x,t\right) =\frac{C}{\sqrt{\pi }}2^{\mu +1/2}\,\Gamma \left( \mu
+1\right) \,t^{\mu }\,\exp \left( -\frac{x^{2}}{8t}\right) \,D_{-1-2\mu
}\left( \frac{x}{\sqrt{2\kappa t}}\right) .
\end{equation*}%
Note that, according to (\ref{D_nu(0)}), the heat flux at $x=0$ is the one
imposed at the boundary, i.e. (\ref{BC_power_law}),
\begin{equation*}
q\left( 0,t\right) =C\,t^{\mu }=q\left( t\right) .
\end{equation*}
\end{enumerate}
\end{example}

\subsection{Error functions}

\begin{theorem}
\label{Theorem: Laplace 4} For $a\in \mathbb{R}\setminus {0}$ and $b\in
\mathbb{C}$, the following convolution integral holds:%
\begin{eqnarray}
&&\int_{0}^{t}\frac{\exp \left( b^{2}\tau \right) \,\mathrm{erfc}\left( b%
\sqrt{\tau }\right) }{\left( t-\tau \right) ^{3/2}}\exp \left( -\frac{a^{2}}{%
t-\tau }\right) \,d\tau  \label{Int_ML_1} \\
&=&\frac{\sqrt{\pi }}{\left\vert a\right\vert }\exp \left( 2\left\vert
a\right\vert b+b^{2}t\right) \,\mathrm{erfc}\left( \frac{\left\vert
a\right\vert }{\sqrt{t}}+b\sqrt{t}\right) .  \notag
\end{eqnarray}
\end{theorem}

\begin{proof}
According to (\ref{L-1[ML]}) and (\ref{ML_1/2,1}), consider the inverse
Laplace transform:
\begin{equation*}
\mathcal{L}^{-1}\left[ \frac{s^{-1/2}}{s^{1/2}+b};t\right] =\mathrm{E}%
_{1/2,1}\left( -b\sqrt{t}\right) =\exp \left( b^{2}t\right) \,\mathrm{erfc}%
\left( b\sqrt{t}\right) ,
\end{equation*}%
as well as the inverse Laplace transforms \cite[Eqns. 2.2.3(6)\& 2.2.1(9)]%
{Prudnikov5}:%
\begin{eqnarray*}
\mathcal{L}^{-1}\left[ \exp \left( -a\sqrt{s}\right) ;t\right] &=&\frac{a}{2%
\sqrt{\pi }t^{3/2}}\exp \left( -\frac{a^{2}}{4t}\right) , \\
\mathcal{L}^{-1}\left[ \frac{s^{-1/2}}{s^{1/2}+b}\exp \left( -a\sqrt{s}%
\right) ;t\right] &=&\exp \left( ab+b^{2}t\right) \,\mathrm{erfc}\left(
\frac{a}{2\sqrt{t}}+b\sqrt{t}\right) ,
\end{eqnarray*}%
Thus, applying the Laplace convolution theorem (\ref{Convolution_Laplace}),
we obtain for $a>0$%
\begin{eqnarray*}
&&\frac{a}{2\sqrt{\pi }}\int_{0}^{t}\frac{\exp \left( b^{2}\tau \right) \,%
\mathrm{erfc}\left( b\sqrt{\tau }\right) }{\left( t-\tau \right) ^{3/2}}\exp
\left( -\frac{a^{2}}{t-\tau }\right) \,d\tau \\
&=&\mathcal{L}^{-1}\left[ \frac{s^{-1/2}}{s^{1/2}+b}\exp \left( -a\sqrt{s}%
\right) ;t\right] =e^{ab+b^{2}t}\,\mathrm{erfc}\left( \frac{a}{2\sqrt{t}}+b%
\sqrt{t}\right) .
\end{eqnarray*}%
Expand the result for $a\in \mathbb{R}\setminus {0}$ to complete the proof.
\end{proof}

\begin{example}
The integral (\ref{Int_ML_1}) appears in heat conduction in a semi-infinite
medium initially at zero temperature with a heat convection on its surface.

\begin{enumerate}
\item Physical setup \newline
Consider a semi-infinite solid ($x\geq 0$) with thermal diffusivity\ $\kappa
$,and thermal conductivity $K$, which is initially at zero temperature. For $%
t>0$, the boundary at $x=0$ is exposed to a convective flux due to a
constant environment temperature $T_{0}$, with $H$ the heat transfer
coefficient between the boundary and the environment.

\item Governing Equations

\begin{itemize}
\item Heat equation
\begin{equation}
\frac{\partial T\left( x,t\right) }{\partial t}=\kappa \frac{\partial
^{2}T\left( x,t\right) }{\partial x^{2}},\quad x>0,\,t>0.
\label{Heat_Equation}
\end{equation}

\item Boundary conditions
\begin{equation}
-\left. \frac{\partial T\left( x,t\right) }{\partial x}\right\vert
_{x=0}=h\,[T_{0}-T(0,t)],\quad h=\frac{H}{K}.  \label{BC_convection}
\end{equation}%
\begin{equation}
\lim_{x\rightarrow \infty }T\left( x,t\right) =0.  \label{BC_at_infinity}
\end{equation}

\item Initial condition
\begin{equation}
T(x,0)=0.  \label{IC_t=0}
\end{equation}
\end{itemize}

\item The surface convection\ can be solved in the Laplace domain as
follows. Apply the Laplace transform to the heat equation (\ref%
{Heat_Equation}), taking into account (\ref{IC_t=0})\ and the derivative
theorem of the Laplace transform (\ref{Derivative_Laplace}), to obtain \
\begin{equation}
s\,\mathcal{L}\left[ T\left( x,t\right) ;s\right] =\kappa \frac{\partial ^{2}%
}{\partial x^{2}}\mathcal{L}\left[ T\left( x,t\right) ;s\right] .
\label{ODE_Laplace}
\end{equation}%
Solve for (\ref{ODE_Laplace}), taking into account (\ref{BC_at_infinity}),
to obtain%
\begin{equation*}
\mathcal{L}\left[ T\left( x,t\right) ;s\right] =\mathcal{L}\left[ T\left(
0,t\right) ;s\right] \exp \left( -\sqrt{\frac{s}{\kappa }}x\right) ,
\end{equation*}%
thus%
\begin{equation}
-\left. \frac{\partial }{\partial x}\mathcal{L}\left[ T\left( x,t\right) ;s%
\right] \right\vert _{x=0}=\mathcal{L}\left[ T\left( 0,t\right) ;s\right] \,%
\sqrt{\frac{s}{\kappa }}.  \label{Ec_1}
\end{equation}%
Now, apply the Laplace transform to the boundary condition (\ref%
{BC_convection}), taking into account that $\mathcal{L}\left[ 1;s\right] =%
\frac{1}{s}$ \cite[Eqn. 1.3]{Schiff}, to obtain%
\begin{equation}
-\left. \frac{\partial }{\partial x}\mathcal{L}\left[ T\left( x,t\right) ;s%
\right] \right\vert _{x=0}=h\left[ \frac{T_{0}}{s}-\mathcal{L}\left[ T\left(
0,t\right) ;s\right] \right] .  \label{Ec_2}
\end{equation}%
From (\ref{Ec_1})\ and (\ref{Ec_2}), we obtain%
\begin{equation*}
\frac{1}{T_{0}}\mathcal{L}\left[ T\left( 0,t\right) ;s\right] =\frac{b}{%
s\left( \sqrt{s}+b\right) },\quad b=h\sqrt{\kappa }.
\end{equation*}%
Apply the inverse Laplace transform given in \cite[Eqn. 2.1.7(28)]%
{Prudnikov5}, to obtain%
\begin{equation}
T\left( 0,t\right) =T_{0}\left[ 1-\exp \left( b^{2}t\right) \,\mathrm{erfc}%
\left( b\sqrt{t}\right) \right] .  \label{BC_T(0,t)}
\end{equation}

\item According to \cite[Sect. 2.5]{Jaeger}, the solution of (\ref%
{Heat_Equation}), (\ref{BC_T(0,t)})\ and (\ref{IC_t=0}) is given by the
convolution integral:
\begin{equation}
T\left( x,t\right) =\frac{a}{\sqrt{\pi }}\int_{0}^{t}\frac{T\left( 0,\tau
\right) }{\left( t-\tau \right) ^{3/2}}\exp \left( -\frac{a^{2}}{t-\tau }%
\right) \,d\tau ,\quad a=\frac{x}{2\sqrt{\kappa }}.
\label{T(x,t)_convolution_Jaeger}
\end{equation}%
Note that, by performing the substitutions $\tau \mapsto t-\tau $ and $%
u^{2}=a^{2}/\tau $, and taking into account the definition of the
complementary error function (\ref{erfc_def}), we have
\begin{eqnarray}
&&\frac{a}{\sqrt{\pi }}\int_{0}^{t}\frac{1}{\left( t-\tau \right) ^{3/2}}%
\exp \left( -\frac{a^{2}}{t-\tau }\right) \,d\tau  \label{Int_heat} \\
&=&\int_{0}^{t}\frac{1}{\tau ^{3/2}}\exp \left( -\frac{a^{2}}{\tau }\right)
\,d\tau  \notag \\
&=&\frac{2}{\sqrt{\pi }}\int_{a/\sqrt{t}}^{\infty }\exp \left( -u^{2}\right)
\,du=\mathrm{erfc}\left( \frac{a}{\sqrt{t}}\right) .  \notag
\end{eqnarray}%
Therefore, inserting (\ref{BC_T(0,t)})\ in (\ref{T(x,t)_convolution_Jaeger}%
),\ and applying (\ref{Int_heat})\ and the convolution integral given in (%
\ref{Int_ML_1}), we arrive at%
\begin{equation}
T\left( x,t\right) =T_{0}\left[ \mathrm{erfc}\left( \frac{x}{2\sqrt{\kappa t}%
}\right) -e^{hx+h^{2}\kappa t}\,\mathrm{erfc}\left( \frac{x}{2\sqrt{\kappa t}%
}+h\sqrt{\kappa t}\right) \right] .  \label{T(x,t)_radiation}
\end{equation}

\item Despite the fact that the boundary value problem given in (\ref%
{Heat_Equation})-(\ref{IC_t=0})\ is solved in \cite[Sect. 2.7]{Jaeger}
without using the Laplace transform, we present here an alternative
derivation to arrive at (\ref{T(x,t)_radiation}).
\end{enumerate}
\end{example}

\begin{corollary}
Applying (\ref{D^n[exp*erfc]}), repeated differentiation of (\ref{Int_ML_1}%
)\ with respect to $b$ leads to the following result for $a\in \mathbb{R}%
\setminus {0}$, $b\in
\mathbb{C}
$, and $n=0,1,2,\ldots $%
\begin{eqnarray}
&&\int_{0}^{t}\frac{\tau ^{n/2}\,H_{-n-1}\left( b\sqrt{\tau }\right) }{%
\left( t-\tau \right) ^{3/2}}\exp \left( -\frac{a^{2}}{t-\tau }\right) d\tau
\label{Convolution_Hermite_n} \\
&=&\frac{\sqrt{\pi }\,t^{n/2}}{\left\vert a\right\vert }\exp \left( -\frac{%
a^{2}}{t}\right) \,H_{-n-1}\left( \frac{\left\vert a\right\vert }{\sqrt{t}}+b%
\sqrt{t}\right) .  \notag
\end{eqnarray}
\end{corollary}

\begin{theorem}
\label{Theorem: Laplace 5} For $a\in
\mathbb{C}
$, the following convolution integrals hold:%
\begin{eqnarray}
\int_{0}^{t}\frac{\exp \left( a^{2}\tau \right) \,\mathrm{erf}\left( a\sqrt{%
\tau }\right) }{\tau \sqrt{t-\tau }}\,d\tau &=&\frac{\pi }{\sqrt{t}}\,%
\mathrm{erfi}\left( a\sqrt{t}\right) ,  \label{Convolution_erfi} \\
\int_{0}^{t}\frac{\exp \left( -a^{2}\tau \right) \,\mathrm{erfi}\left( a%
\sqrt{\tau }\right) }{\tau \sqrt{t-\tau }}\,d\tau &=&\frac{\pi }{\sqrt{t}}\,%
\mathrm{erf}\left( a\sqrt{t}\right) .  \label{Convolution_erf}
\end{eqnarray}
\end{theorem}

\begin{proof}
Apply the Laplace convolution theorem (\ref{Convolution_Laplace}), taking
into account the inverse Laplace transforms \cite[Eqns. 2.1.1(3)\
\&2.5.4.(14)-(15)]{Prudnikov5}:%
\begin{eqnarray*}
\mathcal{L}^{-1}\left[ \frac{1}{\sqrt{s}};t\right] &=&\frac{1}{\sqrt{\pi t}},
\\
\mathcal{L}^{-1}\left[ \log \left( \frac{\sqrt{s}+a}{\sqrt{s}-a}\right) ;t%
\right] &=&\frac{\exp \left( a^{2}t\right) \,\mathrm{erf}\left( a\sqrt{t}%
\right) }{t}, \\
\mathcal{L}^{-1}\left[ \frac{1}{\sqrt{s}}\log \left( \frac{\sqrt{s}+a}{\sqrt{%
s}-a}\right) ;t\right] &=&\sqrt{\frac{\pi }{t}}\,\mathrm{erfi}\left( a\sqrt{t%
}\right) ,
\end{eqnarray*}%
to obtain%
\begin{equation*}
\frac{1}{\sqrt{\pi }}\int_{0}^{t}\frac{\exp \left( a^{2}\tau \right) \,%
\mathrm{erf}\left( a\sqrt{\tau }\right) }{\tau \sqrt{t-\tau }}\,d\tau =%
\mathcal{L}^{-1}\left[ \frac{1}{\sqrt{s}}\log \left( \frac{\sqrt{s}+a}{\sqrt{%
s}-a}\right) ;t\right] =\sqrt{\frac{\pi }{t}}\,\mathrm{erfi}\left( a\sqrt{t}%
\right) ,
\end{equation*}%
thus (\ref{Convolution_erfi})\ is proved. Apply the substitution $a\mapsto
i\,a$ in (\ref{Convolution_erfi})\ and multiply the result by $\frac{1}{i}$,
to obtain (\ref{Convolution_erf}), taking into account the definition of the
imaginary error function (\ref{erfi_def}).

As an alternative proof of (\ref{Convolution_erf}), we proceed as follows.
Taking into account the property (\ref{Dawson_erfi}) and the expansion of
the Dawson function (\ref{Dawson_series}), rewrite the integral given in (%
\ref{Convolution_erf})\ as
\begin{eqnarray}
\int_{0}^{t}\frac{\exp \left( -a^{2}\tau \right) \,\mathrm{erfi}\left( a%
\sqrt{\tau }\right) }{\tau \sqrt{t-\tau }}\,d\tau &=&\frac{2}{\sqrt{\pi }}%
\int_{0}^{t}\frac{\mathrm{daw}\left( a\sqrt{\tau }\right) }{\tau \sqrt{%
t-\tau }}\,d\tau  \label{Int_dawson_1} \\
&=&\frac{2a}{\sqrt{\pi }}\sum_{k=0}^{\infty }\frac{\left( -a^{2}\right) ^{k}%
}{\left( \frac{3}{2}\right) _{k}}\underset{I}{\underbrace{\int_{0}^{t}\frac{%
\tau ^{k-1/2}}{\sqrt{t-\tau }}\,d\tau }}.  \notag
\end{eqnarray}%
The integral $I$ can be recast as a beta integral (\ref{Beta_def})\
performing the substitution $\tau =t\,u$,
\begin{equation}
I=t^{k}\int_{0}^{1}u^{k-1/2}\left( 1-u\right) ^{-1/2}du=t^{k}\frac{\Gamma
\left( k+\frac{1}{2}\right) \,\Gamma \left( \frac{1}{2}\right) }{\Gamma
\left( k+1\right) }=\pi \,t^{k}\,\frac{\left( \frac{1}{2}\right) _{k}}{k!},
\label{I_erfi__resultado}
\end{equation}%
where we have applied (\ref{Gamma(n+1)}) and (\ref{Gamma_particular}), as
well as the definition of the Pochhammer symbol (\ref{Pochhammer_def}).
Substitute back (\ref{I_erfi__resultado})\ in (\ref{Int_dawson_1}), and take
into account the definition of the generalized hypergeometric function (\ref%
{pFq_def}), to obtain:%
\begin{eqnarray*}
\int_{0}^{t}\frac{\exp \left( -a^{2}\tau \right) \,\mathrm{erfi}\left( a%
\sqrt{\tau }\right) }{\tau \sqrt{t-\tau }}\,d\tau &=&2a\,\sqrt{\pi }%
\sum_{k=0}^{\infty }\frac{\left( \frac{1}{2}\right) _{k}\,\left(
-a^{2}t\right) ^{k}}{k!\left( \frac{3}{2}\right) _{k}} \\
&=&2a\,\sqrt{\pi }\,_{1}F_{1}\left(
\begin{array}{c}
\frac{1}{2} \\
\frac{3}{2}%
\end{array}%
;-a^{2}t\right) .
\end{eqnarray*}%
Finally, apply the reduction formula given in (\ref{1F1_reduction_4b})\ and
the definition of the imaginary error function (\ref{erfi_def})\ to complete
the proof of (\ref{Convolution_erf}).
\end{proof}

\begin{example}
The integral obtained in (\ref{Convolution_erf})\ arises in the heat
conduction in a semi-infinite medium initially at zero temperature, with a
prescribed heat flux at the surface. The heat flux models coupled processes
of kinetic decay and exponential thermal dissipation.

\begin{enumerate}
\item Physical setup \newline
Consider a semi-infinite medium of thermal diffusivity $\kappa $, and
thermal conductivity $K$, located at $x\geq 0$, which is initially at zero
temperature. Consider that the surface of the medium (i.e. at $x=0$) is
exposed to the following heat flux, for $t>0$,
\begin{equation}
q\left( t\right) =q\exp \left( -a^{2}t\right) \,\frac{\mathrm{erfi}\left(
\,a\,\sqrt{t}\right) }{t}=\frac{2q}{\sqrt{\pi }}\,\frac{\mathrm{daw}\left( a%
\sqrt{t}\right) }{t}.  \label{q(t)_erfi}
\end{equation}%
The exponential term $\exp \left( -a^{2}\tau \right) $ represents an actual
physical dissipation due to a pulsed laser ablation. When a beam strikes a
surface, it generates an ionized gas cloud (plasma). This plasma absorbs a
fraction of the incident laser energy via inverse Bremsstrahlung.
Consequently, the net energy that successfully penetrates the plasma and
reaches the actual surface of the solid decreases exponentially over time
\cite{Marla}. The term with the imaginary error function $\mathrm{erfi}%
\left( \,a\,\sqrt{t}\right) /t$ typically arises in first-order chemical
reaction-diffusion systems \cite{Bard}. Note that $q\left( t\right) $ is
physically admissible since, according to (\ref{Dawson_series}), we have
\begin{equation*}
q\left( t\right) \approx \frac{q}{\sqrt{t}},\quad t\rightarrow 0.
\end{equation*}%
This matches exactly the behavior of an instantaneous impact heat source
(Dirac delta function) or the onset of a thermal pulse where the surface
spatial gradient is infinitely steep. Also, according to (\ref{Dawson_z->inf}%
), we have
\begin{equation*}
q\left( t\right) \approx \frac{q}{a\sqrt{\pi }}t^{-3/2},\quad t\rightarrow
\infty .
\end{equation*}%
This physically models the complete depletion of surface reactants, allowing
the solid to cool down by pure internal conduction into the semi-infinite
medium.

\item Governing equations

\begin{itemize}
\item Heat equation
\begin{equation}
\frac{\partial T\left( x,t\right) }{\partial t}=\kappa \frac{\partial
^{2}T\left( x,t\right) }{\partial x^{2}},\quad x>0,\,t>0.
\label{Heat_Equation_erfi}
\end{equation}

\item Boundary conditions
\begin{equation}
-K\left. \frac{\partial T\left( x,t\right) }{\partial x}\right\vert
_{x=0}=q(t),  \label{BC_erfi}
\end{equation}%
\begin{equation}
\lim_{x\rightarrow \infty }T\left( x,t\right) =0.  \label{BC_infinity_erfi}
\end{equation}

\item Initial condition
\begin{equation}
T(x,0)=0.  \label{IC_erfi}
\end{equation}
\end{itemize}

\item According to \cite[Sect. 2.9]{Jaeger}, the solution of (\ref%
{Heat_Equation_erfi})-(\ref{IC_erfi}) is given by the convolution integral:
\begin{equation*}
T\left( x,t\right) =\frac{\sqrt{\kappa }}{\sqrt{\pi }K}\int_{0}^{t}\frac{%
q\left( \tau \right) }{\sqrt{t-\tau }}\exp \left( -\frac{b^{2}}{t-\tau }%
\right) \,d\tau ,\quad b=\frac{x}{2\sqrt{\kappa }},
\end{equation*}%
thus, taking into account the flux given in (\ref{q(t)_erfi}), we have%
\begin{equation*}
T\left( 0,t\right) =\frac{q\sqrt{\kappa }}{\sqrt{\pi }K}\int_{0}^{t}\frac{%
\exp \left( -a^{2}\tau \right) \,\mathrm{erfi}\left( \,a\,\sqrt{\tau }%
\right) }{\tau \sqrt{t-\tau }}\,d\tau .
\end{equation*}%
Therefore, applying the convolution integral given in (\ref{Convolution_erf}%
), we obtain:%
\begin{equation*}
T\left( 0,t\right) =\frac{q\sqrt{\pi \kappa }}{K\sqrt{t}}\,\mathrm{erf}%
\left( a\sqrt{t}\right) .
\end{equation*}%
Note that, according to (\ref{erf_series}), the surface temperature at the
beginning of the process is bounded since
\begin{equation*}
\lim_{t\rightarrow 0}T\left( 0,t\right) =2a\frac{q\sqrt{\kappa }}{K}<\infty .
\end{equation*}%
Also, taking into account (\ref{erf_asymptotic}), we have
\begin{equation*}
T\left( 0,t\right) \approx \frac{q\sqrt{\pi \kappa }}{K\sqrt{t}},\quad
t\rightarrow \infty ,
\end{equation*}%
thus in the long term, the applied heat has diffused into the interior of
the semi-infinite solid, cooling the surface via pure internal
one-dimensional conduction towards the infinite sink.
\end{enumerate}
\end{example}

\begin{corollary}
Perform the change of variables $\tau =t\,v^{2}$ and $z=a\sqrt{t}$ in (\ref%
{Convolution_erfi}) and (\ref{Convolution_erf}), and apply (\ref{Dawson_erfi}%
) to obtain the following integral representations for $z\in
\mathbb{C}
$,
\begin{eqnarray}
\mathrm{erfi}\left( z\right) &=&\frac{2}{\pi }\int_{0}^{1}\frac{\exp \left(
z^{2}v^{2}\right) \,\mathrm{erf}\left( z\,v\right) }{v\sqrt{1-v^{2}}}\,dv,
\label{erfi_integral} \\
\mathrm{erf}\left( z\right) &=&\frac{4}{\pi ^{3/2}}\int_{0}^{1}\frac{\mathrm{%
daw}\left( z\,v\right) }{v\sqrt{1-v^{2}}}\,dv.  \label{erf_integral}
\end{eqnarray}
\end{corollary}

\begin{theorem}
\label{Theorem: Laplace 6} For $a\in
\mathbb{C}
$ and $\Re \left( \nu \right) \geq -\frac{1}{2}$, the following convolution
integral holds true:\
\begin{eqnarray}
&&\int_{0}^{t}\frac{\tau ^{\nu }}{t-\tau }\,_{1}F_{1}\left(
\begin{array}{c}
\frac{1}{2} \\
\nu +1%
\end{array}%
;a^{2}\tau \right) \,\mathrm{erfi}\left( a\sqrt{t-\tau }\right) \,d\tau
\label{Int_1F1_erfi} \\
&=&\frac{2a\,t^{\nu +1/2}\,\Gamma \left( \nu +1\right) }{\Gamma \left( \nu +%
\frac{3}{2}\right) }\,_{2}F_{2}\left(
\begin{array}{c}
1,1 \\
\frac{3}{2},\nu +\frac{3}{2}%
\end{array}%
;a^{2}t\right) .  \notag
\end{eqnarray}
\end{theorem}

\begin{proof}
Apply the Laplace convolution theorem (\ref{Convolution_Laplace}), taking
into account the following inverse Laplace transforms (see \cite[Eqns.
2.6.1(8)\&(11)]{Prudnikov5})%
\begin{eqnarray*}
\mathcal{L}^{-1}\left[ \arcsin \left( a/\sqrt{s}\right) ;t\right] &=&\frac{%
\mathrm{erfi}\left( a\sqrt{t}\right) }{2t}, \\
\mathcal{L}^{-1}\left[ \frac{\arcsin \left( a/\sqrt{s}\right) }{s^{\nu }%
\sqrt{s-a^{2}}};t\right] &=&\frac{a\,t^{\nu }}{\Gamma \left( \nu +1\right) }%
\,_{2}F_{2}\left(
\begin{array}{c}
1,1 \\
\frac{3}{2},\nu +1%
\end{array}%
;a^{2}t\right) ,
\end{eqnarray*}%
and \cite[Eqn. 2.1.2(1)]{Prudnikov5}:%
\begin{equation*}
\mathcal{L}^{-1}\left[ \frac{1}{s^{\nu }\sqrt{s-a^{2}}};t\right] =\frac{%
t^{\nu -1/2}}{\Gamma \left( \nu +\frac{1}{2}\right) }\,_{1}F_{1}\left(
\begin{array}{c}
\frac{1}{2} \\
\frac{1}{2}+\nu%
\end{array}%
;a^{2}t\right) .
\end{equation*}
\end{proof}

\begin{corollary}
\label{Corollary Laplace 6}Take $\nu =0$ in (\ref{Int_1F1_erfi})\ and apply (%
\ref{1F1_reduction_3})\ to obtain:\
\begin{eqnarray}
&&\int_{0}^{t}e^{-a\,\tau }\,I_{0}\left( a\left( t-\tau \right) \right) \,%
\mathrm{erfi}\left( \sqrt{2a\,\tau }\right) \,\frac{d\tau }{\tau }
\label{Convolution_I0*erfi} \\
&=&4\sqrt{\frac{2a\,t}{\pi }}\,e^{-a\,t}\,_{2}F_{2}\left(
\begin{array}{c}
1,1 \\
\frac{3}{2},\frac{3}{2}%
\end{array}%
;2a\,t\right) ,  \notag
\end{eqnarray}%
or equivalently, according to (\ref{Dawson_erfi}),
\begin{eqnarray}
&&\int_{0}^{t}e^{-a\,\tau /2}\,I_{0}\left( \frac{a}{2}\,\tau \right) \,\frac{%
\mathrm{daw}\left( \sqrt{a\,\left( t-\tau \right) }\right) }{t-\tau }\,d\tau
\label{Convolution_I0_daw} \\
&=&2\sqrt{a\,t}\,e^{-a\,t}\,_{2}F_{2}\left(
\begin{array}{c}
1,1 \\
\frac{3}{2},\frac{3}{2}%
\end{array}%
;a\,t\right) .
\end{eqnarray}
\end{corollary}

\begin{example}
The integral obtained in (\ref{Convolution_I0_daw})\ arises in the heat
conduction in a semi-infinite medium initially at zero temperature, with a
prescribed heat flux at the surface and constant lateral heat loss to the
environment.

\begin{enumerate}
\item Physical setup \newline
Consider a semi-infinite medium of thermal diffusivity $\kappa $, and
thermal conductivity $K$, located at $x\geq 0$, which is initially at zero
temperature. Also, consider that the surface of the medium (i.e. at $x=0$)
is exposed to the following heat flux, for $t>0$,
\begin{equation}
q\left( t\right) =q\,\frac{\mathrm{erf}\left( \sqrt{a\,t}\right) }{\sqrt{t}}.
\label{q(t)_erf}
\end{equation}%
Note that the asymptotic behaviour of $q\left( t\right) $ is realistic since
it does not diverge for $t\rightarrow 0$ and $t\rightarrow \infty $. Indeed,
according to (\ref{erf_series}), we have
\begin{equation*}
q\left( t\right) \approx \frac{2q}{\sqrt{\pi }},\quad t\rightarrow 0,
\end{equation*}%
and, according to (\ref{erf_asymptotic}), we have
\begin{equation*}
q\left( t\right) \approx \frac{q}{\sqrt{t}},\quad t\rightarrow \infty .
\end{equation*}

\item Governing equations

\begin{itemize}
\item Heat equation
\begin{equation}
\frac{\partial T\left( x,t\right) }{\partial t}=\kappa \frac{\partial
^{2}T\left( x,t\right) }{\partial x^{2}}-a\,T\left( x,t\right) ,\quad
x>0,\,t>0,  \label{Heat_Equation_I0}
\end{equation}%
where $a$ represents the linear cooling coefficient.

\item Boundary conditions
\begin{equation}
-K\left. \frac{\partial T\left( x,t\right) }{\partial x}\right\vert
_{x=0}=q\left( t\right) ,  \label{BC_q(t)_I0}
\end{equation}%
\begin{equation}
\lim_{x\rightarrow \infty }T\left( x,t\right) =0.  \label{BC_inf_I0}
\end{equation}

\item Initial condition
\begin{equation}
T(x,0)=0.  \label{IC_I0}
\end{equation}
\end{itemize}

\item Apply the Laplace transform to the heat equation (\ref%
{Heat_Equation_I0}), taking into account the derivative theorem of the
Laplace transform (\ref{Derivative_Laplace})\ and the initial condition (\ref%
{IC_I0}), to obtain%
\begin{equation*}
\left( s+a\right) \,\mathcal{L}\left[ T\left( x,t\right) ,s\right] =\kappa
\frac{\partial ^{2}}{\partial x^{2}}\mathcal{L}\left[ T\left( x,t\right) ,s%
\right] .
\end{equation*}%
Thus, taking into account the boundary condition (\ref{BC_inf_I0}), we solve
for $\mathcal{L}\left[ T\left( x,t\right) ,s\right] $, to obtain%
\begin{equation}
\mathcal{L}\left[ T\left( x,t\right) ,s\right] =\mathcal{L}\left[ T\left(
0,t\right) ,s\right] \exp \left( -\sqrt{\frac{s+a}{\kappa }}x\right) .
\label{Laplace_I0}
\end{equation}%
Now, apply the Laplace transform to the boundary condition (\ref{BC_q(t)_I0}%
),
\begin{equation}
-K\left. \frac{\partial }{\partial x}\mathcal{L}\left[ T\left( x,t\right) ,s%
\right] \right\vert _{x=0}=\mathcal{L}\left[ q(t),s\right] =Q(s),
\label{Laplace_flux_1}
\end{equation}%
where, inserting the heat flux given in (\ref{q(t)_erf})\ and applying the
Laplace transform \cite[Eqn. 3.7.1(10)]{Prudnikov4}, we have%
\begin{equation}
Q\left( s\right) =q\,\mathcal{L}\left[ \frac{\mathrm{erf}\left( \,\sqrt{a\,t}%
\right) }{\sqrt{t}},s\right] =\frac{2q}{\sqrt{\pi \,s}}\arctan \left( \sqrt{%
\frac{a}{s}}\right) .  \label{Q(s)_def}
\end{equation}%
However, from (\ref{Laplace_I0}), we have%
\begin{equation}
-K\left. \frac{\partial }{\partial x}\mathcal{L}\left[ T\left( x,t\right) ,s%
\right] \right\vert _{x=0}=E\,\mathcal{L}\left[ T\left( 0,t\right) ,s\right]
\,\sqrt{s+a},  \label{Laplace_flux_2}
\end{equation}%
where $E=\frac{K}{\sqrt{\kappa }}$ denotes the effusivity of the
semi-infinite medium. Match the results (\ref{Laplace_flux_1})\ and (\ref%
{Laplace_flux_2})\ to obtain%
\begin{equation*}
\mathcal{L}\left[ T\left( 0,t\right) ,s\right] =\frac{1}{E}\frac{Q\left(
s\right) }{\sqrt{s+a}}=\frac{1}{E}\frac{G\left( s\right) }{\sqrt{s\left(
s+a\right) }},
\end{equation*}%
where, according to (\ref{Q(s)_def}), we have%
\begin{equation*}
G\left( s\right) =\sqrt{s}\,Q\left( s\right) =\frac{2q}{\sqrt{\pi }}\arctan
\left( \sqrt{\frac{a}{s}}\right) .
\end{equation*}%
Taking into account the inverse Laplace transform \cite[Eqn. 2.6.4(14)]%
{Prudnikov5}, and the property (\ref{Dawson_erfi}), we have%
\begin{equation}
\mathcal{L}^{-1}\left[ G\left( s\right) ,t\right] =\frac{q}{\sqrt{\pi }t}%
\exp \left( -a\,t\right) \,\mathrm{erfi}\left( \sqrt{a\,t}\right) =\frac{2q}{%
\pi \,t}\,\mathrm{daw}\left( \sqrt{a\,t}\right) .  \label{L-1=daw}
\end{equation}%
Also, according to the inverse Laplace transform \cite[Eqn. 2.1.5(3)]%
{Prudnikov5}, we have%
\begin{eqnarray}
\mathcal{L}^{-1}\left[ \frac{1}{\sqrt{s\left( s+a\right) }},t\right] &=&%
\mathcal{L}^{-1}\left[ \frac{1}{\sqrt{\left( s+\frac{a}{2}\right)
^{2}-\left( \frac{a}{2}\right) ^{2}}},t\right]  \label{L-1=exp*I0} \\
&=&e^{-at/2}\,\mathcal{L}^{-1}\left[ \frac{1}{\sqrt{s^{2}-\left( \frac{a}{2}%
\right) ^{2}}},t\right] =e^{-at/2}\,I_{0}\left( \frac{a}{2}t\right) .  \notag
\end{eqnarray}%
Next, apply the convolution theorem (\ref{Convolution_Laplace}), taking into
account the inverse Laplace transforms (\ref{L-1=daw})\ and (\ref{L-1=exp*I0}%
) to obtain%
\begin{equation*}
T\left( 0,t\right) =\frac{2\,q}{\pi \,E}\int_{0}^{t}e^{-a\,\tau
/2}\,I_{0}\left( \frac{a}{2}\,\tau \right) \,\frac{\mathrm{daw}\left( \sqrt{%
a\,\left( t-\tau \right) }\right) }{t-\tau }\,d\tau .
\end{equation*}%
Finally, apply the convolution integral (\ref{Convolution_I0_daw}) to obtain
the evolution of the temperature on the surface,
\begin{equation*}
T\left( 0,t\right) =\frac{4\,q}{\pi \,E}\,\sqrt{a\,t}\,e^{-a\,t}\,_{2}F_{2}%
\left(
\begin{array}{c}
1,1 \\
\frac{3}{2},\frac{3}{2}%
\end{array}%
;a\,t\right) .
\end{equation*}
\end{enumerate}
\end{example}

\section{Conclusions \label{Section: Conclusions}}

Using the Laplace convolution theorem and tabulated inverse Laplace
transforms, we compute several convolution integrals involving exponential,
error, and parabolic cylinder functions, which are not reported in the most
common tables of integrals in Theorems \ref{Theorem: Laplace 1}, \ref%
{Theorem: Laplace 2}, \ref{Theorem: Laplace 3}, \ref{Theorem: Laplace 4}, %
\ref{Theorem: Laplace 5}, and \ref{Theorem: Laplace 6}. Particular cases of
these Theorems are given in corollaries \ref{Corollary a Laplace 3}, \ref%
{Corollary b Laplace 3} and \ref{Corollary Laplace 6}. We check that these
convolution integrals are not computable with the aid of the symbolic
integration engine of Mathematica. In addition, we expanded the above
results using the repeated integration formula to obtain (\ref%
{Int_repeated_integration}), as well as repeated parametric differentiation,
yielding (\ref{Convolution_erfc_n})-(\ref{Proof_1})\ and (\ref%
{Convolution_Hermite_n}). It is noteworthy that parametric differentiation
has been used previously in \cite{Hetnarski} to generate inverse Laplace
transforms of exponential form similar to the ones used in this paper.

It is worth noting that we derive alternative proofs without using the
Laplace transform whenever this is possible. However, since these
alternative proofs are highly non-trivial, we conclude that the Laplace
convolution theorem is a very powerful tool for computing Laplace
convolution integrals straightforwardly.

Applications of the numerous convolution integrals obtained, to heat
transfer and linear viscoelasticity, are given in the examples presented
throughout the paper. It is noteworthy that most of the examples presented
are novel in the existing literature.

As a by-product, we obtain new integral representations of the error,
complementary error, and imaginary error functions in (\ref{erf_integral}), (%
\ref{erfc_integral_representation})\ and (\ref{erfi_integral}), respectively.

Finally, all numerical checks were carried out with Mathematica. The
corresponding Mathematica notebook is available at %
\url{https://shorturl.at/pK2Cn}.


\end{document}